\documentclass{amsart}

\usepackage{amsfonts}
\usepackage{graphicx}
\usepackage[onehalfspacing]{setspace}
\usepackage{hyperref}
\usepackage[ngerman, english]{babel}
\usepackage{datetime}
\usepackage{enumitem} 
\usepackage{mathtools}
\usepackage{mathrsfs}
\usepackage{cite}
\usepackage{url}
\usepackage[a4paper,margin=20ex]{geometry}
\usepackage{graphicx}
\usepackage{xpatch}
\usepackage{amsthm}
\usepackage{mleftright}
\usepackage{microtype}

\allowdisplaybreaks[1]

\xpatchcmd{\proof}
  {\itshape}
  {\bfseries}
  {}
  {}

\DeclarePairedDelimiterX\Set[1]{\lbrace}{\rbrace}%
 {  #1 }

\makeatletter
\def\bign#1{\mathclose{\hbox{$\left#1\vbox to8.5\p@{}\right.\n@space$}}\mathopen{}}
\def\Bign#1{\mathclose{\hbox{$\left#1\vbox to8.5\p@{}\right.\n@space$}}\mathopen{}}
\def\biggn#1{\mathclose{\hbox{$\left#1\vbox to14.5\p@{}\right.\n@space$}}\mathopen{}}
\def\Biggn#1{\mathclose{\hbox{$\left#1\vbox to17.5\p@{}\right.\n@space$}}\mathopen{}}
\makeatother

\newcommand\blfootnote[1]{%
  \begingroup
  \renewcommand\thefootnote{}\footnote{#1}%
  \addtocounter{footnote}{-1}%
  \endgroup
}

\makeatletter
\newenvironment{proof*}[1][\proofname]{\par
  \pushQED{\qed}%
  \normalfont \partopsep=\z@skip \topsep=\z@skip
  \trivlist
  \item[\hskip\labelsep
        \itshape
    #1\@addpunct{.}]\ignorespaces
}{%
  \popQED\endtrivlist\@endpefalse
}
\makeatother
\makeatletter
\newcommand*{\rom}[1]{\expandafter\@slowromancap\romannumeral #1@}
\makeatother
\def\Xint#1{\mathchoice
{\XXint\displaystyle\textstyle{#1}}%
{\XXint\textstyle\scriptstyle{#1}}%
{\XXint\scriptstyle\scriptscriptstyle{#1}}%
{\XXint\scriptscriptstyle\scriptscriptstyle{#1}}%
\!\int}
\def\XXint#1#2#3{{\setbox0=\hbox{$#1{#2#3}{\int}$ }
\vcenter{\hbox{$#2#3$ }}\kern-.582\wd0}}

\def\dashint{\Xint-}

\DeclareMathOperator*{\essinf}{ess\,inf}

\usepackage{graphicx}				
								
\usepackage{amssymb}
\newtheorem*{defin}{Definition}
\newtheorem{thm}{Theorem}[section]
\newtheorem{cor}[thm]{Corollary}
\newtheorem{lem}[thm]{Lemma}
\newtheorem*{remark}{Remark}
\newtheorem*{remarkb}{Remark on boundary regularity}

\newtheorem{prop}[thm]{Proposition}

\title{$H^{s,p}$ regularity theory for a class of nonlocal elliptic equations}
\author{Simon Nowak}
\address{Universit\"at Bielefeld, Fakult\"at f\"ur Mathematik, Postfach 100131, D-33501 Bielefeld, Germany}
\email{simon.nowak@uni-bielefeld.de}
\begin{document}

\maketitle
\begin{abstract}
In this paper, we study the regularity of weak solutions to a class of nonlocal elliptic equations in Bessel potential spaces $H^{s,p}$. Our main results can be seen as an extension of the well-known $W^{1,p}$ regularity theory for local second-order elliptic equations in divergence form to the nonlocal setting.
\end{abstract}
\pagestyle{headings}

\tableofcontents

\section{Introduction} 
\subsection{Basic setting}
In this work, we study the regularity of weak solutions to nonlocal elliptic equations of the form \blfootnote{Supported by SFB 1283 of the German Research Foundation.}
\begin{equation} \label{nonlocaleq}
L_A u +bu = \sum_{i=1}^m L_{D_i} g_i + f \text{ in } \Omega \subset \mathbb{R}^n
\end{equation}
in Bessel potential spaces $H^{s,p}$. Roughly speaking, the purpose of this paper is to prove the implication $u \in H^{s,2} \implies u \in H^{s,p}$ for the whole range of exponents $p \in (2,\infty)$ in the case of possibly very irregular data.
Here $s \in (0,1)$, $\Omega \subset \mathbb{R}^n$ $(n > 2s)$ is a domain (= open set), $b,f,g_i:\mathbb{R}^n \to \mathbb{R}$ ($i=1,...,m,$ $m \in \mathbb{N}$) are given functions and 
$$ L_Au(x) = 2 \lim_{\varepsilon \to 0} \int_{\mathbb{R}^n \setminus B_\varepsilon(x)} \frac{A(x,y)}{|x-y|^{n+2s}} (u(x)-u(y))dy, \quad x \in \Omega,$$
is a nonlocal operator. Furthermore, the function $A:\mathbb{R}^n \times \mathbb{R}^n \to \mathbb{R}$ is measurable and we assume that there exists a constant $\lambda \geq 1$ such that
\begin{equation} \label{eq1}
\lambda^{-1} \leq A(x,y) \leq \lambda \text{ for almost all } x,y \in \mathbb{R}^n.
\end{equation}
Moreover, we require $A$ to be symmetric, i.e.
\begin{equation} \label{symmetry}
A(x,y)=A(y,x) \text{ for almost all } x,y \in \mathbb{R}^n.
\end{equation}
We call such a function $A$ a kernel coefficient. We define $\mathcal{L}_0(\lambda)$ as the class of all such measurable kernel coefficients $A$ that satisfy the conditions (\ref{eq1}) and (\ref{symmetry}). Note that in our main results, we additionally assume that $A$ is translation invariant, cf. section 1.3. Moreover, throughout this work $D_i:\mathbb{R}^n \times \mathbb{R}^n \to \mathbb{R}$ ($i=1,...,m$) are assumed to be measurable functions that are symmetric and bounded by some $\Lambda>0$, i.e.
\begin{equation} \label{bound}
\sum_{i=1}^m |D_i(x,y)| \leq \Lambda  \text{ for almost all } x,y \in \mathbb{R}^n.
\end{equation}
Define the spaces
$$H^s(\Omega | \mathbb{R}^n)= \left \{u:\mathbb{R}^n \to \mathbb{R} \text{ measurable } \mathrel{\Big|} \int_{\Omega}u(x)^2dx+ \int_{\Omega} \int_{\mathbb{R}^n} \frac{(u(x)-u(y))^2}{|x-y|^{n+2s}}dydx < \infty \right \}$$
and
$$ H^s_0(\Omega | \mathbb{R}^n)= \left \{ u \in H^s(\Omega | \mathbb{R}^n) \mid u =0 \text{ a.e. in } \mathbb{R}^n \setminus \Omega \right \}.$$
For all measurable functions $u,\varphi:\mathbb{R}^n \to \mathbb{R}$ we define the bilinear form associated to the operator $L_A$ by
$$ \mathcal{E}_A(u,\varphi) = \int_{\mathbb{R}^n} \int_{\mathbb{R}^n} \frac{A(x,y)}{|x-y|^{n+2s}} (u(x)-u(y))(\varphi(x)-\varphi(y))dydx,$$
provided that the above expression is well-defined and finite, this is e.g. the case if $u \in H^s(\Omega | \mathbb{R}^n)$ and $\varphi \in H^s_0(\Omega | \mathbb{R}^n)$.
Analogously we consider the bilinear forms $\mathcal{E}_{D_i}(u,\varphi)$ associated to the operators $L_{D_i}$.
\begin{defin}
	Given $b \in L^\infty(\Omega)$, $f \in L^2(\Omega)$ and $g_i \in H^s(\Omega | \mathbb{R}^n)$, $i=1,...,m$, we say that $u \in H^s(\Omega | \mathbb{R}^n)$ is a weak solution to the equation $L_A u + bu = \sum_{i=1}^m L_{D_i} g_i+ f$ in $\Omega$, if 
	$$ \mathcal{E}_A(u,\varphi) + (bu,\varphi)_{L^2(\Omega)}= \sum_{i=1}^m \mathcal{E}_{D_i}(g_i,\varphi) + (f,\varphi)_{L^2(\Omega)} \quad \forall \varphi \in H^s_0(\Omega | \mathbb{R}^n).$$
\end{defin}

\subsection{Some previous results}

Studying the regularity of weak solutions to equations of the form (\ref{nonlocaleq}) has been a very active area of research in recent years. Results concerning H\"older regularity were e.g. obtained in \cite{Kassmann}, \cite{finnish}, \cite{Silvestre}, \cite{Peral}, \cite{FracLap} and \cite{NonlocalGeneral}, while results concerning higher differentiability in Sobolev spaces were e.g. obtained in \cite{Cozzi} and \cite{BL}. Regarding higher integrability, in \cite{Bass} and \cite{Auscher} it was shown that under the assumptions from section 1.1 there exists some small $\sigma>0$, such that for any weak solution $u$ of $L_A u=f$ in $\mathbb{R}^n$ the function
\begin{equation} \label{nonlocgrad}
\nabla^s u(x)= \left ( \int_{\mathbb{R}^n} \frac{(u(x)-u(y))^2}{|x-y|^{n+2s}}dy \right )^{\frac{1}{2}}
\end{equation}
belongs to $L^{2+\sigma}(\mathbb{R}^n)$ whenever $f \in L^2(\mathbb{R}^n)$. In view of a classical characterization of Bessel potential spaces due to Stein (cf. Theorem \ref{altcharBessel}), this actually implies that $u$ belongs to the Bessel potential space $H^{s,2+\sigma}(\mathbb{R}^n)$. Similar results were proved in \cite{selfimpro}, \cite{Schikorra} and \cite{Auscher}, where it was shown that under the assumptions from section 1.1 $u$ actually not only possesses a higher integrability but also a slightly higher differentiabillity.

\subsection{Main results}

The aim of this work is to prove the $H^{s,p}$ regularity for solutions $u$ to equations of the form (\ref{nonlocaleq}) not only for some $p>2$ close enough to $2$, but for the full range $p \in (2,\infty)$. 
In order to accomplish this, we restrict our attention to the following class of translation invariant kernel coefficients.
\begin{defin}
	Let $\lambda\geq 1$. We say that a kernel coefficient $A \in \mathcal{L}_0(\lambda)$ belongs to the class $\mathcal{L}_1(\lambda)$, if there exists a measurable function $a:\mathbb{R}^n \to \mathbb{R}$ such that $A(x,y)=a(x-y)$ for all $x,y \in \mathbb{R}^n$, that is, if $A$ is translation invariant.
\end{defin}

Our main result concerning local regularity in Bessel potential spaces $H^{s,p}$ is the following.
\begin{thm} \label{mainint5}
\emph{(Local $H^{s,p}$ regularity in domains)} \newline
Let $\Omega \subset \mathbb{R}^n$ be a bounded domain, $p \in (2,\infty)$, $s \in (0,1)$, $b \in L^\infty(\Omega)$, $g_i \in H^s(\Omega | \mathbb{R}^n) \cap H^{s,p}(\mathbb{R}^n)$ and $f \in L^{p_\star}(\Omega)$, where $p_\star=\max \left \{\frac{pn}{n+ps},2 \right \}$. 
If $A$ belongs to the class $\mathcal{L}_1(\lambda)$ and if all $D_i$ are symmetric and satisfy (\ref{bound}), then for any weak solution $u \in H^s(\Omega | \mathbb{R}^n)$ 
of the equation
\begin{equation} \label{nonloceq1}
L_A u + bu= \sum_{i=1}^m L_{D_i} g_i + f \text{ in } \Omega
\end{equation}
we have $u \in H^{s,p}_{loc}(\Omega)$ and $u \in W^{s,p}_{loc}(\Omega)$.
\end{thm}

\begin{remark} \normalfont
We actually obtain a slightly stronger result (cf. Theorem \ref{mainint3}) than the one given by Theorem \ref{mainint5} in terms of certain function spaces $H^{s,p}(\Omega | \mathbb{R}^n)$ that generalize the space $H^s(\Omega | \mathbb{R}^n)$ to the case when $p \neq 2$, cf. section 3. By a useful alternative characterization of Bessel potential spaces (cf. Theorem \ref{altcharBessel}), this space $H^{s,p}(\Omega | \mathbb{R}^n)$ is actually contained in $H^{s,p}(\Omega)$ whenever $\Omega$ is regular enough, so that this result then implies Theorem \ref{mainint5}. Moreover, since Theorem \ref{mainint5} is concerned with local regularity, the above result remains true if we generalize the notion of weak solutions to an appropriate notion of local weak solutions, cf. section 7.
\end{remark}

\begin{remark} \normalfont
Although in Theorem \ref{mainint5} and in our other main results we are primarily concerned with regularity in Bessel potential spaces $H^{s,p}$, due to the classical embedding $H^{s,p} \hookrightarrow W^{s,p}$ for $p \in [2,\infty)$ (cf. Proposition \ref{Sobolevrelations}), we also obtain regularity in Sobolev-Slobodeckij spaces $W^{s,p}$.
\end{remark}

If the equation is posed on the whole space $\mathbb{R}^n$, we are actually able to establish the following global regularity result.

\begin{thm} \label{mainint0}
	\emph{($H^{s,p}$ regularity on the whole space $\mathbb{R}^n$)} \newline
	Let $p \in (2,\infty)$, $s \in (0,1)$, $b \in L^\infty(\mathbb{R}^n)$, $g_i \in H^{s}(\mathbb{R}^n) \cap H^{s,p}(\mathbb{R}^n)$ and $f \in L^2(\mathbb{R}^n) \cap L^{p_\star}(\mathbb{R}^n)$, where $p_\star=\max \left \{\frac{pn}{n+ps},2 \right \}$. 
	If $A$ belongs to $\mathcal{L}_1(\lambda)$ and if all $D_i$ are symmetric and satisfy (\ref{bound}), then for any weak solution $u \in H^s(\mathbb{R}^n)$ 
	of the equation 
	\begin{equation} \label{nonloceq6}
	L_A u + bu= \sum_{i=1}^m L_{D_i} g_i + f \text{ in } \mathbb{R}^n
	\end{equation}
	we have $u \in H^{s,p}(\mathbb{R}^n)$ and $u \in W^{s,p}(\mathbb{R}^n)$. 
\end{thm}

In the case when $b=g_i=0$ $(i=1,...,m)$, by writing $A(x,y)=a(x-y)$, Theorem \ref{mainint0} can also be deduced by applying \cite[Theorem 1]{Bogdan} to the symbol
$$ M(\xi)=\frac{\int_{\mathbb{R}^n}(\cos(\xi \cdot y)-1)a^{-1}(y)V(dy)}{\int_{\mathbb{R}^n}(\cos(\xi \cdot y)-1)V(dy)}, \quad V(dy)=\frac{a(y)}{|y|^{n+2s}}dy.$$

The main achievement of this paper is that we develop the $H^{s,p}$ regularity theory for weak solutions of the equation (\ref{nonlocaleq}) for a general right-hand side and especially in the setting of arbitrary domains $\Omega \subset \mathbb{R}^n$. This is because although in the case of local elliptic equations local regularity in domains $\Omega$ can be deduced from the corresponding result in $\mathbb{R}^n$ by using a cutoff argument, in the nonlocal setting such a cutoff argument requires an additional assumption on the solution in the complement of $\Omega$ (cf. \cite{Warma} or \cite{KassMengScott}), which is not required in Theorem \ref{mainint5}. Another advantage of our approach is that it can be generalized to an appropriate class of nonlinear nonlocal equations. We plan to do this in a future work. 

We also want to mention that in \cite{DongKim} a somewhat related result was proved. Building on an approach used by Krylov in \cite{Krylov} to obtain $W^{1,p}$ estimates for local second-order elliptic equations, the authors obtained $H^{2s,p}$ a priori estimates for strong solutions of the equation $L_A=f$ in $\mathbb{R}^n$.

In order to prove our main results, we instead apply a variation of another approach commonly used in order to obtain $W^{1,p}$ regularity results for local elliptic equations in divergence form which we briefly describe in section 1.4. This enables us to simultaneously treat the problems of local $H^{s,p}$ regularity in domains $\Omega$ and the problem of global $H^{s,p}$ regularity on the whole space $\mathbb{R}^n$.

\subsection{$W^{1,p}$ regularity theory for second-order elliptic equations in divergence form}
Let us briefly review the well-known $W^{1,p}$ regularity theory for local second-order elliptic equations in divergence form treated for example in \cite{ByunLp}, \cite{Weighted12} or \cite{ByunReifenberg}, where the authors build on an approach first introduced by Caffarelli and Peral in \cite{CaffarelliPeral}. Consider the equation 
\begin{equation} \label{76}
\textnormal{div} (B \nabla u) = \textnormal{div} g + f \quad \text{in } \Omega ,
\end{equation}
where $\Omega \subset \mathbb{R}^n$ is a domain, the matrix of coefficients $B:\mathbb{R}^n \to \mathbb{R}^{n\times n}$, $B(x)= (B_{ij}(x))_{i,j=1,...,n}$ has measurable entries, is uniformly elliptic and bounded, while $g: \Omega \rightarrow \mathbb{R}^n$ and $f:\Omega \to \mathbb{R}$ are given functions. Furthermore, solutions are understood in an appropriate weak sense, cf. \cite{Gilbarg} or \cite{ByunLp}. Let $2<p<\infty$. A natural question corresponding to Theorem \ref{mainint5} to ask in this context is the following: Under which assumptions on $B,f$ and $g$ does any weak solution $u \in H^1_{loc}(\Omega)$ of (\ref{76}) in fact belong to the space $u \in W^{1,p}_{loc}(\Omega)$? The minimal assumptions on $g$ and $f$ for this property to hold are $g \in L^p_{loc}(\Omega,\mathbb{R}^n)$ and $f \in L^{\frac{np}{n+p}}_{loc}(\Omega)$, while the coefficients $B_{ij}$ are required to have small enough BMO-seminorms, cf. \cite{ByunLp}. The strategy to obtain such local $W^{1,p}$ estimates used e.g. in \cite{ByunLp}, \cite{Weighted12} or \cite{ByunReifenberg} is as follows. One approximates the gradient of the weak solution $u$ of (\ref{76}) in $L^2$ by the gradient of a weak solution $v$ of a suitable equation $\textnormal{div} (B_0 \nabla u)=0$, where $B_0$ has constant coefficients. One then uses the fact that $v$ satisfies a local $C^{0,1}$ estimate along with a real-variable argument based on the Vitali covering lemma, the Hardy-Littlewood maximal function and an alternative characterization of $L^p$ spaces in order to obtain an $L^p$ estimate for $\nabla u$, which in view of interpolation then implies the desired local $W^{1,p}$ estimate. 

The main idea of our approach in the nonlocal setting is to apply similar arguments with the gradient $\nabla u$ replaced by the nonlocal s-gradient $\nabla^s u$ defined in (\ref{nonlocgrad}). However, due to the nonlocal nature of the operator $\nabla^s$ and the equations we consider, in our setting we have to overcome a number of difficulties that are not present in the local case. \nolinebreak
\subsection{Some notation}
For convenience, let us fix some notation which we use throughout the paper. By $C$ and $C_i$, $i \in \mathbb{N}$, we always denote positive constants, while dependences on parameters of the constants will be shown in parentheses. As usual, by
$$ B_r(x_0):= \{x \in \mathbb{R}^n \mid |x-x_0|<r \}$$
we denote the open ball with center $x_0 \in \mathbb{R}^n$ and radius $r>0$. Moreover, if $E \subset \mathbb{R}^n$ is measurable, then by $|E|$ we denote the $n$-dimensional Lebesgue-measure of $E$. If $0<|E|<\infty$, then for any $u \in L^1(E)$ we define
$$ \overline u_{E}:= \dashint_{E} u(x)dx := \frac{1}{|E|} \int_{E} u(x)dx.$$

\section{Some tools from real analysis}

In this section, we discuss some results from real analysis that will play key roles in our treatment of the $H^{s,p}$ regularity theory for nonlocal elliptic equations.

The following result can be proved by using the well-known Vitali covering lemma, cf. \cite[Theorem 2.7]{ByunLp}.

\begin{lem} \label{modVitali}
Assume that $E$ and $F$ are measurable sets in $\mathbb{R}^n$ that satisfy $E \subset F \subset B_1$. Assume further that 
there exists some $\varepsilon \in (0,1)$ such that 
$$
|E| < \varepsilon |B_1|,
$$
and that for all $x \in B_1$ and any $r \in (0,1)$ with $|E \cap B_r(x)| \geq \varepsilon |B_r(x)|$ we have 
$$B_r(x) \cap B_1 \subset F. $$
Then we have 
$$
|E| \leq 10^n \varepsilon |F|.
$$
\end{lem}

Another tool we use is the Hardy-Littlewood maximal function.

\begin{defin}
Let $f \in L^1_{loc}(\mathbb{R}^n)$. Then the Hardy-Littlewood maximal function \newline $\mathcal{M} f: \mathbb{R}^n \to [0,\infty]$ of $f$ is defined by 
$$ \mathcal{M} f(x):=\mathcal{M} (f)(x) := \sup_{\rho>0} \dashint_{B_\rho(x)} |f(y)|dy .$$
Moreover, for any domain $\Omega \subset \mathbb{R}^n$ and any function $f \in L^1(\Omega)$, consider the zero extension of $f$ to $\mathbb{R}^n$  
$$ f_{\Omega} (x) := \begin{cases} 
f(x) \text{, if } x \in \Omega \\
0 \quad  \text{ , if } x \notin \Omega.
\end{cases} $$
We then define 
$$ \mathcal{M}_{\Omega} f := \mathcal{M} f_\Omega. $$
\end{defin}

Rather straightforward but important features of the Hardy-Littlewood maximal function are its scaling and translation invariance, given by the following Lemma which can be proved by using a change of variables.
\begin{lem} \label{scale}
Let $f \in L^1_{loc}(\mathbb{R}^n)$, $r>0$ and $y \in \mathbb{R}^n$. Then for the function $f_{r,y}(x):=f(rx+y)$ and any $x \in \mathbb{R}^n$ we have 
$$\mathcal{M} f_{r,y}(x) = \mathcal{M} f(rx+y) .$$
Similarly, for any domain $\Omega \subset \mathbb{R}^n$, any function $f \in L^1(\Omega)$ and any $x \in \Omega$ we have 
$$\mathcal{M}_{ \Omega^\prime} f_{r,y}(x) = \mathcal{M}_\Omega f(rx+y), $$
where $\Omega^\prime:= \{ \frac{x-y}{r} \mid x \in \Omega \}$. 
\end{lem}

We remark that for any $f \in L^1_{loc}(\mathbb{R}^n)$, $\mathcal{M} f$ is Lebesgue-measurable.
Intuitively, the Hardy-Littlewood maximal function of a function $f$ in general seems to be much larger than the function $f$ itself. However, the following results show that when measured appropriately, the size of $\mathcal{M} f$ can actually be controlled by the size of $f$, cf. \cite{SteinSingular}.
\begin{thm} \label{Maxfun} 
Let $\Omega \subset \mathbb{R}^n$ be a domain.
\begin{enumerate}[label=(\roman*)]
\item (weak 1-1 estimate) If $f \in L^1(\Omega)$ and $t>0$, then 
$$
| \{x \in \Omega \mid \mathcal{M}_\Omega(f)(x) > t \}| \leq \frac{C}{t} \int_{\Omega} |f| dx, 
$$
where $C=C(n)>0$.
\item (strong p-p estimates) If $f \in L^p(\Omega)$ for some $p \in (1,\infty]$, then 
$$
||f||_{L^p(\Omega)} \leq ||\mathcal{M}_\Omega f||_{L^p(\Omega)} \leq C ||f||_{L^p(\Omega)}, 
$$
where $C=C(n,p)>0$.
\item If $f \in L^p(\Omega)$ for some $p \in [1,\infty]$, then the function $\mathcal{M}_\Omega f$ is finite almost everywhere.
\end{enumerate}
\end{thm}

We conclude this section by giving an alternative characterization of $L^p$ spaces, cf. \cite[Lemma 7.3]{CaffFully}. It can be proved by using the well-known formula
$$ ||f||^p_{L^p(\Omega)} = p \int_0^{\infty} t^{p-1} \left | \left \{ x \in \Omega \mid f(x) > t \right \} \right | dt. $$
\begin{lem} \label{Caff}
Let $0< p<\infty$. Furthermore, suppose that $f$ is a nonnegative and measurable function in a bounded domain $\Omega \subset \mathbb{R}^n$ and let $\tau > 0$, $\beta>1$ be constants. Then for
$$ S:= \sum_{k = 1}^{\infty}{\beta^{kp} | \{ x \in \Omega \mid f(x) > \tau \beta^k \} | } , $$
we have
$$
C^{-1} S \leq ||f||^p_{L^p(\Omega)} \leq C (|\Omega|+S)
$$
for some constant $C=C(\tau,\beta,p) > 0$.
In particular, we have $f \in L^p(\Omega)$ if and only if $S<\infty$.
\end{lem}

\section{Fractional Sobolev spaces and the s-gradient}
We start this section by defining a first type of fractional Sobolev spaces which is probably the most widely used type of such spaces in the literature concerned with elliptic equations.
\begin{defin}
Let $\Omega \subset \mathbb{R}^n$ be a domain. For $p \in [1,\infty)$ and $s \in (0,1)$, we define the Sobolev-Slobodeckij space
$$W^{s,p}(\Omega):=\left \{u \in L^p(\Omega) \mathrel{\Big|} \int_{\Omega} \int_{\Omega} \frac{|u(x)-u(y)|^p}{|x-y|^{n+sp}}dydx<\infty \right \}$$
with norm
$$ ||u||_{W^{s,p}(\Omega)} := \left (\int_{\Omega}|u(x)|^p dx + \int_{\Omega} \int_{\Omega} \frac{|u(x)-u(y)|^p}{|x-y|^{n+sp}}dydx \right )^{1/p} .$$
Moreover, we define the corresponding local Sobolev-Slobodeckij spaces by
$$ W^{s,p}_{loc}(\Omega):= \left \{ u \in L^p_{loc}(\Omega) \mid u \in W^{s,p}(\Omega^\prime) \text{ for any domain } \Omega^\prime \subset \subset \Omega \right \}.$$
Finally, set
$$ H^s(\Omega):=W^{s,2}(\Omega).$$
\end{defin}
\begin{remark} \normalfont
	The space $H^s(\Omega)$ is a Hilbert space with respect to the inner product
	\begin{equation} \label{inner}
	 (u,v)_{H^s(\Omega)}:= (u,v)_{L^2(\Omega)} + \int_{\Omega} \int_{\Omega} \frac{(u(x)-u(y))(v(x)-v(y))}{|x-y|^{n+2s}}dydx. 
	 \end{equation}
\end{remark}

We use the following fractional Poincar\'e inequality, cf. \cite[Lemma 3.10]{DRV}.
\begin{lem} \label{Poincare}
Let $\Omega \subset \mathbb{R}^n$ be a bounded Lipschitz domain and $s \in (0,1)$. For any $u \in H^s(\Omega)$ we have
$$ \int_{\Omega} \left | u(x)- \overline u_{\Omega} \right |^2 dx \leq C \int_{\Omega} \int_{\Omega} \frac{|u(x)-u(y)|^2}{|x-y|^{n+2s}}dydx,$$
where $C=C(s,\Omega)>0$.
\end{lem}
We also use the following type of fractional Sobolev spaces.
\begin{defin}
For $p \in [1,\infty)$ and $s \in \mathbb{R}$, consider the Bessel potential space
$$ H^{s,p}(\mathbb{R}^n):=\left \{u \in L^p(\mathbb{R}^n) \mid \mathcal{F}^{-1} \left [ \left ( 1+|\xi|^2 \right )^{\frac{s}{2}} \mathcal{F} f \right ] \in L^p(\mathbb{R}^n) \right \} ,$$
where $\mathcal{F}$ denotes the Fourier transform and $\mathcal{F}^{-1}$ denotes the inverse Fourier transform. We equip $H^{s,p}(\mathbb{R}^n)$ with the norm
$$ ||u||_{H^{s,p}(\mathbb{R}^n)} := \left | \left |\mathcal{F}^{-1} \left [ \left ( 1+|\xi|^2 \right )^{\frac{s}{2}} \mathcal{F} f \right ] \right | \right |_{L^p(\mathbb{R}^n)}.$$
Moreover, for any domain $\Omega \subset \mathbb{R}^n$ we define 
$$H^{s,p}(\Omega):= \left \{ u \big|_\Omega \mid u \in H^{s,p}(\mathbb{R}^n) \right \} $$
with norm 
$$ ||u||_{H^{s,p}(\Omega)} := \inf \left \{ ||v||_{H^{s,p}(\mathbb{R}^n)} \mid v \big |_\Omega = u \right \}.$$
Furthermore, we define the corresponding local Bessel potential spaces by 
$$ H^{s,p}_{loc}(\Omega) := \left \{ u \in L^p_{loc}(\Omega) \mid u \in H^{s,p}(\Omega^\prime) \text{ for any domain } \Omega^\prime \subset \subset \Omega \right \}.$$   
\end{defin}

The following result gives some relations between Bessel potential spaces and Sobolev-Slobodeckij spaces.

\begin{prop} \label{Sobolevrelations}
Let $\Omega \subset \mathbb{R}^n$ be a domain.
\begin{enumerate} [label=(\roman*)]
\item If $\Omega$ is a bounded Lipschitz domain or $\Omega=\mathbb{R}^n$, then for all $s \in (0,1)$, $p \in (1,2]$ we have $W^{s,p}(\Omega) \hookrightarrow H^{s,p}(\Omega)$.
\item For any $s \in (0,1)$ and any $p \in [2,\infty)$ we have $ H^{s,p}(\Omega) \hookrightarrow W^{s,p}(\Omega)$.
\end{enumerate}
\end{prop}
For a proof of Proposition \ref{Sobolevrelations}, we refer to Theorem 5 in chapter \rom{5} of \cite{SteinSingular} for the case when $\Omega = \mathbb{R}^n$. For general domains $\Omega$, part $(i)$ then follows by extending an arbitrary function $u \in W^{s,p}(\Omega)$ to a function that belongs to $W^{s,p}(\mathbb{R}^n)$, for which an additional assumption on $\Omega$ is required, cf. \cite[Theorem 5.4]{Hitch}. Part $(ii)$ follows similarly by extending an arbitrary function $u \in H^{s,p}(\Omega)$ to a function that belongs to $H^{s,p}(\mathbb{R}^n)$, which by definition of $H^{s,p}(\Omega)$ is possible for arbitrary domains.

We now define a function that can be viewed as a nonlocal analogue to the euclidean norm of the gradient of a function in the local context.
\begin{defin}
Let $s \in (0,1)$. For any domain $\Omega \subset \mathbb{R}^n$ and any measurable function $u:\Omega \to \mathbb{R}$, we define the s-gradient $\nabla^s_\Omega u:\Omega \to [0,\infty]$ by 
$$ \nabla^s_\Omega u(x):= \left ( \int_{_\Omega} \frac{(u(x)-u(y))^2}{|x-y|^{n+2s}}dy \right )^{\frac{1}{2}}.$$
Moreover, for any measurable function $u:\mathbb{R}^n \to \mathbb{R}$ we define $\nabla^s u:=\nabla^s_{\mathbb{R}^n} u$.
\end{defin}

In view of Proposition \ref{Sobolevrelations}, for any bounded Lipschitz domain $\Omega$ we have $u \in H^{s,2}(\Omega)$ if and only if $u \in L^2(\Omega)$ and $\nabla^s_\Omega u \in L^2(\Omega)$.
The following result shows that a similar alternative characterization of Bessel potential spaces in terms of the $s$-gradient is also true for a much wider range of exponents $p$.
This characterization was first given by Stein in \cite{SteinBessel} in the case when $\Omega=\mathbb{R}^n$. For the case when $\Omega$ is an arbitrary Lipschitz domain we refer to \cite[Theorem 1.3]{Soler} or \cite[Theorem 2.10]{Soler1}, where this characterization is proved in the more general context of Triebel-Lizorkin spaces and so-called uniform domains.
\begin{thm} \label{altcharBessel}
Let $s \in (0,1)$, $p \in \left (\frac{2n}{n+2s},\infty \right )$ and assume that $\Omega \subset \mathbb{R}^n$ is a bounded Lipschitz domain or that $\Omega=\mathbb{R}^n$. Then we have $u \in H^{s,p}(\Omega)$ if and only if $u \in L^p(\Omega)$ and $\nabla^s_\Omega u \in L^p(\Omega)$. Moreover, we have
$$ ||u||_{H^{s,p}(\Omega)} \simeq ||u||_{L^p(\Omega)} + ||\nabla^s_\Omega u||_{L^p(\Omega)} $$
in the sense of equivalent norms.
\end{thm}

We remark that the above result holds in particular for any $p \geq 2$.

Even though we primarily work in some domain $\Omega \subset \mathbb{R}^n$, we obtain most results in this work in terms of the global s-gradient $\nabla^s$ instead of the localized s-gradient $\nabla^s_\Omega$, which is mostly due to the nonlocal character of the equations we consider. In order to state our main result in domains in an optimal way (cf. Theorem \ref{mainint3}), we therefore also define the following natural nonstandard function spaces.
\begin{defin}
Let $\Omega \subset \mathbb{R}^n$ be a domain. For $p \in [1,\infty)$ and $s \in (0,1)$, we define the linear space
$$H^{s,p}(\Omega | \mathbb{R}^n):=\left \{u:\mathbb{R}^n \to \mathbb{R} \text{ measurable } \mid u \in L^p(\Omega) \text{ and } \nabla^s u \in L^p(\Omega) \right \}.$$
Moreover, we define the corresponding local spaces by 
$$ H^{s,p}_{loc}(\Omega | \mathbb{R}^n) = \left \{ u:\mathbb{R}^n \to \mathbb{R} \text{ measurable } \mid u \in H^{s,p}(\Omega^\prime | \mathbb{R}^n) \text{ for any domain } \Omega^\prime \subset \subset \Omega \right \}.$$
Also, we use the spaces
$$ H^{s,p}_0(\Omega | \mathbb{R}^n):= \left \{ u \in H^{s,p}(\Omega | \mathbb{R}^n) \mid u =0 \text{ a.e. in } \mathbb{R}^n \setminus \Omega \right \}.$$
Furthermore, set $$H^s(\Omega | \mathbb{R}^n):=H^{s,2}(\Omega | \mathbb{R}^n), \text{ } H^s_{loc}(\Omega | \mathbb{R}^n):=H^{s,2}_{loc}(\Omega | \mathbb{R}^n) \text{ and } H^s_0(\Omega | \mathbb{R}^n):=H^{s,2}_0(\Omega | \mathbb{R}^n).$$
\end{defin}
%Note that the definitions of $H^s(\Omega | \mathbb{R}^n)$ and $H^s_0(\Omega | \mathbb{R}^n)$ match the corresponding definitions from the introduction.
\begin{remark} \normalfont
Since for any $u \in H^{s}_0(\Omega | \mathbb{R}^n)$ we have
$$ \int_{\mathbb{R}^n}u(x)^2 dx + \int_{\mathbb{R}^n} \int_{\mathbb{R}^n} \frac{(u(x)-u(y))^2}{|x-y|^{n+2s}}dydx \leq \int_{\Omega}u(x)^2 dx + 2 \int_{\Omega}(\nabla^s u(x))^2 dx < \infty,$$
$H^{s}_0(\Omega | \mathbb{R}^n)$ clearly is a closed subspace of $H^s(\mathbb{R}^n)$ and thus also a Hilbert space with respect to the inner product $(u,v)_{H^s(\mathbb{R}^n)}$ defined in (\ref{inner}).
\end{remark}
\begin{remark} \normalfont
In view of Theorem \ref{altcharBessel}, for any bounded Lipschitz domain $\Omega \subset \mathbb{R}^n$ and all $s \in (0,1)$, $p \in \left (\frac{2n}{n+2s},\infty \right )$ we have the inclusions
$$ H^{s,p}(\mathbb{R}^n) \subset H^{s,p}(\Omega | \mathbb{R}^n) \subset H^{s,p}(\Omega).$$
In the case when $\Omega \subset \mathbb{R}^n$ is an arbitrary domain this implies the inclusions
$$ H^{s,p}(\mathbb{R}^n) \subset H^{s,p}_{loc}(\Omega | \mathbb{R}^n) \subset H^{s,p}_{loc}(\Omega).$$
\end{remark}
We also use the following embedding theorems of Bessel potential spaces. Parts $(i)$ and $(ii)$ follow from \cite[Remark 1.96 $(iii)$]{Triebel}, while the last two parts follow from the corresponding embeddings of Sobolev-Slobodeckij spaces (cf. \cite{Hitch}) and part $(ii)$ of Proposition \ref{Sobolevrelations}.
\begin{thm} \label{BesselEmbedding}
Let $1<p \leq p_1 < \infty$, $s,s_1 \geq 0$ and assume that $\Omega \subset \mathbb{R}^n$ is a domain.
\begin{enumerate} [label=(\roman*)]
\item If $sp<n$, then for any $q \in [p,\frac{np}{n-sp}]$ we have
$$ H^{s,p}(\Omega) \hookrightarrow L^{q}(\Omega).$$
\item More generally, if $ s - \frac{n}{p} = s_1 - \frac{n}{p_1}, $ then
$$ H^{s,p}(\Omega) \hookrightarrow H^{s_1,p_1}(\Omega).$$
\item If $sp = n$, then for any $q \in [p,\infty)$ we have
$$ H^{s,p}(\Omega) \hookrightarrow L^{q}(\Omega).$$
\item If $sp > n$, then we have $$H^{s,p}(\Omega) \hookrightarrow C^\alpha(\Omega),$$
where $\alpha = s-\frac{n}{p}$.
\end{enumerate}
\end{thm}
\section{Some preliminary regularity results}
For the rest of this paper, we fix real numbers $s \in (0,1)$ and $\lambda \geq 1$.
\subsection{$L^\infty$ estimates}
The following Lemma relates the nonlocal tail of a function that often appears naturally in the literature to the $L^2$ norm of its $s$-gradient.
\begin{lem} \label{tailestimate}
For all $r,R>0$ and any $u \in H^s(B_R | \mathbb{R}^n)$ we have
\begin{equation} \label{cool}
\int_{\mathbb{R}^n \setminus B_r} \frac{u(y)^2}{|y|^{n+2s}}dy \leq C(||\nabla^s u||_{L^2(B_R)}^2+||u||_{L^2(B_R)}^2) ,
\end{equation}
where $C=C(n,s,r,R)>0$.
\end{lem}
\begin{proof}
First of all, integration in polar coordinates yields
\begin{equation} \label{intpolar1}
\int_{\mathbb{R}^n \setminus B_r} \frac{dz}{|z|^{n+2s}} = \omega_n \int_r^\infty \frac{\rho^{n-1}}{\rho^{n+2s}}d\rho = \frac{\omega_n}{2s r^{2s}}=:C_1<\infty,
\end{equation}
where $\omega_n$ denotes the surface area of the $n-1$ dimensional unit sphere $\mathbb{S}^{n-1}$. By the Cauchy-Schwarz inequality, Fubini's theorem and (\ref{intpolar1}) we have 
\begingroup
\allowdisplaybreaks
\begin{align*}
\int_{\mathbb{R}^n \setminus B_r} \frac{u(y)^2}{|y|^{n+2s}}dy & = \int_{\mathbb{R}^n \setminus B_r} \frac{\left (u(y)- \dashint_{B_R} u(x)dx + \dashint_{B_R} u(x)dx \right )^2}{|y|^{n+2s}}dy \\
& \leq 2 \left ( \int_{\mathbb{R}^n \setminus B_r} \frac{\left ( \dashint_{B_R} (u(x)-u(y)) dx \right )^2}{|y|^{n+2s}}dy + \int_{\mathbb{R}^n \setminus B_r} \frac{\left ( \dashint_{B_R} u(x) dx \right )^2}{|y|^{n+2s}}dy \right ) \\
& \leq 2 \left (  \int_{\mathbb{R}^n \setminus B_r} \dashint_{B_R} \frac{ (u(x)-u(y))^2 }{|y|^{n+2s}}dxdy + \int_{\mathbb{R}^n \setminus B_r} \frac{\dashint_{B_R} u^2(x) dx}{|y|^{n+2s}}dy \right ) \\
& = \frac{2}{|B_R|} \left ( \int_{B_R} \int_{\mathbb{R}^n \setminus B_r} \frac{ (u(x)-u(y))^2 }{|y|^{n+2s}}dydx + C_1 \int_{B_R} u^2(x) dx \right ).
\end{align*}
\endgroup
Moreover, since for any $x \in B_R$ and any $y \in \mathbb{R}^n \setminus B_r$ we have
$$ |x-y| \leq |x|+|y| < R+|y| = \left ( \frac{R}{|y|}+1 \right ) |y| \leq \left ( \frac{R}{r}+1 \right )|y|,$$
we see that 
$$ \int_{B_R} \int_{\mathbb{R}^n \setminus B_r} \frac{ (u(x)-u(y))^2 }{|y|^{n+2s}}dydx \leq C_2 \int_{B_R} \int_{\mathbb{R}^n \setminus B_r} \frac{ (u(x)-u(y))^2 }{|x-y|^{n+2s}}dydx \leq C_2 \int_{B_R} (\nabla^s u)^2(x) dx,$$
where $C_2=\left ( \frac{R}{r}+1 \right )^{n+2s}$. By combining the above computations, we see that (\ref{cool}) holds with $C=\frac{2}{|B_R|} \max\{C_2,C_1\}$.
\end{proof}

We also use the following local $L^\infty$ estimate for weak solutions to homogeneous nonlocal equations, cf. \cite[Theorem 1.1]{finnish}. We remark that although in \cite{finnish} the below result is stated under the stronger assumption that $u \in H^s(\mathbb{R}^n)$, an inspection of the proof shows that this is not necessary.
\begin{thm} \label{finnish}
Consider a kernel coefficient $A \in \mathcal{L}_0(\lambda)$. For all $0<r<R<\infty$ and any weak solution $u \in H^s(B_R | \mathbb{R}^n)$ of the equation 
$$ L_A u = 0 \text{ in } B_R$$
we have the estimate
$$ ||u||_{L^\infty(B_r)} \leq C \left (\int_{\mathbb{R}^n \setminus B_r} \frac{|u(y)|}{|y|^{n+2s}}dy + ||u||_{L^2(B_R)} \right ),$$
where $C=C(n,s,r,R,\lambda)>0$.
\end{thm}
By combining the above two results, we obtain the following.
\begin{cor} \label{supest}
Consider a kernel coefficient $A \in \mathcal{L}_0(\lambda)$. For all $0<r<R<\infty$ and any weak solution $u \in H^s(B_R | \mathbb{R}^n)$ of the equation 
$$ L_A u = 0 \text{ in } B_R$$
we have the estimate
\begin{equation} \label{gradest81}
||u||_{L^\infty(B_r)} \leq C(||\nabla^s u||_{L^2(B_R)}+||u||_{L^2(B_R)}),
\end{equation}
where $C=C(n,s,r,R,\lambda)>0$.
\end{cor}

\begin{proof}
By Theorem \ref{finnish}, (\ref{intpolar1}) and Lemma \ref{tailestimate} we have
\begin{align*}
||u||_{L^\infty(B_r)} & \leq C_1 \left (\int_{\mathbb{R}^n \setminus B_r} \frac{|u(y)|}{|y|^{n+2s}}dy + ||u||_{L^2(B_R)} \right ) \\
& \leq C_1 \left (C_2^{\frac{1}{2}} \left ( \int_{\mathbb{R}^n \setminus B_r} \frac{u(y)^2}{|y|^{n+2s}}dy \right )^{\frac{1}{2}} + ||u||_{L^2(B_R)} \right ) \\
& \leq C_1 \left(C_2^{\frac{1}{2}}C_3^{\frac{1}{2}}+1 \right) \left (||\nabla^s u||_{L^2(B_R)}+||u||_{L^2(B_R)} \right ),
\end{align*}
where $C_1$ is given by Theorem \ref{finnish}, $C_2$ is given by (\ref{intpolar1}) and $C_3$ is given by Lemma \ref{tailestimate}. This proves (\ref{gradest81}) with $C=C_1 \left(C_2^{\frac{1}{2}}C_3^{\frac{1}{2}}+1 \right)$.
\end{proof}

\begin{cor} \label{taillinf}
Consider a kernel coefficient $A \in \mathcal{L}_0(\lambda)$. For all $0<r<R<\infty$ and any weak solution $u \in H^s(B_R | \mathbb{R}^n)$ of the equation 
$$ L_A u = 0 \text{ in } B_R$$
we have the estimate
\begin{equation} \label{gradest45}
||\nabla^s_{\mathbb{R}^n \setminus B_R} u||_{L^\infty(B_r)} \leq C||\nabla^s u||_{L^2(B_R)},
\end{equation}
where $C=C(n,s,r,R,\lambda)>0$.
\end{cor}

\begin{proof}
	For any $x \in B_r$ and any $y \in \mathbb{R}^n \setminus B_R$ we have
	$$|y| \leq |x-y|+|x| < |x-y|+R=\left(1+\frac{R}{|x-y|} \right )|x-y| \leq \left(1+\frac{R}{R-r} \right) |x-y|.$$
	For almost every $x \in B_r$, it follows that
	\begin{align*}
	& \int_{\mathbb{R}^n \setminus B_R} \frac{(u(x)-u(y))^2}{|x-y|^{n+2s}}dy \\
	\leq & C_1 \int_{\mathbb{R}^n \setminus B_R} \frac{(u(x)-u(y))^2}{|y|^{n+2s}}dy \\
	\leq & 2 C_1 \left ( \int_{\mathbb{R}^n \setminus B_R} \frac{u(x)^2}{|y|^{n+2s}}dy + \int_{\mathbb{R}^n \setminus B_R} \frac{u(y)^2}{|y|^{n+2s}}dy \right ) \\
	\leq & 2 C_1 \left (C_2 ||u||_{L^\infty(B_r)}^2 + C_3 \left ( \int_{B_R} \int_{\mathbb{R}^n} \frac{ (u(z)-u(y))^2}{|z-y|^{n+2s}}dydz + \int_{B_R} u^2(z) dz \right ) \right ) \\
	\leq & 2 C_1 (C_2 C_4+C_3) \left ( \int_{B_R} \int_{\mathbb{R}^n} \frac{ (u(z)-u(y))^2}{|z-y|^{n+2s}}dydz + \int_{B_R} u^2(z) dz \right ) ,
	\end{align*}
	where $C_1:=\left(1+\frac{R}{R-r} \right)^{n+2s}$, $C_2=C_2(n,s,R)$ is given as in (\ref{intpolar1}) in the proof of Lemma \ref{tailestimate}, while $C_3=C_3(n,s,R)$ is given by Lemma \ref{tailestimate} and $C_4=C_4(n,s,\lambda,r,R)$ is given by Corollary \ref{supest}. Set $C_5:=2 C_1 (C_2 C_4+C_3)$.
	Since the function $u- \overline u_{B_R} \in H^s(B_R | \mathbb{R}^n)$ also solves the equation 
	$$ L_{A}(u- \overline u_{B_R})=0 \text{ weakly in } B_R,$$
	the above estimate also applies to the function $u- \overline u_{B_R}$, so that together with the fractional Poincar\'e inequality (Lemma \ref{Poincare}) for almost every $x \in B_r$ we deduce
	\begin{align*}
	|\nabla^s_{\mathbb{R}^n \setminus B_R} u(x)|^2 & = \int_{\mathbb{R}^n \setminus B_R} \frac{(u(x)-u(y))^2}{|x-y|^{n+2s}}dy = \int_{\mathbb{R}^n \setminus B_R} \frac{((u(x)-\overline u_{B_R})-(u(y)-\overline u_{B_R}))^2}{|x-y|^{n+2s}}dy \\
	& \leq C_5 \left (\int_{B_R} \int_{\mathbb{R}^n} \frac{((u(z)-\overline u_{B_R})-(u(y)-\overline u_{B_R}))^2}{|z-y|^{n+2s}}dydz + \int_{B_R} (u(z)-\overline u_{B_R})^2dz \right ) \\
	& \leq C_5 \left (\int_{B_R} \int_{\mathbb{R}^n} \frac{(u(z)-u(y))^2}{|z-y|^{n+2s}}dydz + C_6 \int_{B_R} \int_{ B_R} \frac{(u(z)-u(y))^2}{|z-y|^{n+2s}}dydz \right ) \\
	& \leq C_7 ||\nabla^s u||^2_{L^2(B_R)} ,
	\end{align*}
	where $C_6=C_6(n,s,R)$ and $C_7:=C_5(1+C_6)$, which proves (\ref{gradest45}) with $C=C_7^{\frac{1}{2}}$.
\end{proof}
\subsection{Higher H\"older regularity}
In the basic case when $A \in \mathcal{L}_0(\lambda)$, it can be shown that any weak solution to the corresponding homogeneous nonlocal equation is $C^\alpha$ for some $\alpha>0$, cf. \cite[Theorem 1.2]{finnish}. 
The following result shows that if $A$ is of class $\mathcal{L}_1(\lambda)$, then weak solutions to the corresponding homogeneous nonlocal equation enjoy better H\"older regularity than in general.

\begin{thm} \label{modC2sreg}
	Consider a kernel coefficient $A \in \mathcal{L}_1(\lambda)$ and assume that $u \in H^s(B_5 | \mathbb{R}^n)$ is a weak solution of the equation $L_{A} u = 0$ in $B_5$. Then for any $0<\alpha<\min\{2s,1\}$ we have
	$$ [u]_{C^{\alpha}(B_3)} \leq C ||\nabla^s u||_{L^2(B_5)} ,$$
	where $C=C(n,s,\lambda,\alpha)>0$ and 
	$$[u]_{C^{\alpha}(B_3)}:=\sup_{\substack{_{x,y \in B_3}\\{x \neq y}}} \frac{|u(x)-u(y)|}{|x-y|^{\alpha}}.$$
\end{thm}

We will derive Theorem \ref{modC2sreg} from the following analogue of \cite[Theorem 5.2]{BLS}, where a corresponding result is proved for weak solutions to the fractional $p$-Laplace equation. 

\begin{thm} \label{HiHol}
	Consider a kernel coefficient $A \in \mathcal{L}_1(\lambda)$ and  assume that $u \in H^s(B_5 | \mathbb{R}^n)$ is a weak solution of the equation $L_{A} u = 0$ in $B_5$. Then for any $0<\alpha<\min\{2s,1\}$ we have
	$$ [u]_{C^{\alpha}(B_3)} \leq C \left (||u||_{L^\infty(B_4)} + \int_{B_4} \int_{B_4} \frac{|u(x)-u(y)|^2}{|x-y|^{n+2s}}dydx  + \int_{\mathbb{R}^n \setminus B_4} \frac{|u(y)|}{|y|^{n+2s}}dy \right ),$$
	where $C=C(n,s,\lambda,\alpha)>0$.
\end{thm}

Since the proof in \cite{BLS} is done only in the case when $A(x,y) \equiv 1$ but naturally applies to the setting of arbitrary kernel coefficients $A \in \mathcal{L}_1(\lambda)$, let us briefly explain the modifications that are necessary in order to prove the result in this more general setting. 
Fix $0<r<R$, $h \in \mathbb{R}^n \setminus \{0\}$ such that $|h| \leq \frac{R-r}{2}$ and a test function $\varphi \in H^s_0(B_{(R+r)/2}|\mathbb{R}^n)$. Moreover, suppose that $u \in H^s(B_R | \mathbb{R}^n)$ is a weak solution of 
\begin{equation} \label{tieq}
L_{A} u = 0 \text{ in } B_R. 
\end{equation}
Since the function $\varphi_{-h}(x):=\varphi(x-h)$ belongs to $H^s_0(B_{R}|\mathbb{R}^n)$, we can use $\varphi_{-h}$ as a test function in (\ref{tieq}). Setting $u_{h}(x):=u(x+h)$, along with a change of variables and the assumption that $A \in \mathcal{L}_1(\lambda)$ this yields
\begin{equation} \label{trest}
\begin{aligned}
0=& \int_{\mathbb{R}^n} \int_{\mathbb{R}^n} \frac{A(x,y)}{|x-y|^{n+2s}} (u(x)-u(y))(\varphi_{-h}(x)-\varphi_{-h}(y))dydx \\
= & \int_{\mathbb{R}^n} \int_{\mathbb{R}^n} \frac{A(x+h,y+h)}{|x-y|^{n+2s}} (u_h(x)-u_h(y))(\varphi(x)-\varphi(y))dydx \\
= & \int_{\mathbb{R}^n} \int_{\mathbb{R}^n} \frac{A(x,y)}{|x-y|^{n+2s}} (u_h(x)-u_h(y))(\varphi(x)-\varphi(y))dydx.\\
\end{aligned}
\end{equation}
Moreover, testing (\ref{tieq}) with $\varphi$ yields
\begin{equation} \label{trest1}
\int_{\mathbb{R}^n} \int_{\mathbb{R}^n} \frac{A(x,y)}{|x-y|^{n+2s}} (u(x)-u(y))(\varphi(x)-\varphi(y))dydx=0.
\end{equation}
By subtracting (\ref{trest1}) from (\ref{trest}) and dividing by $|h|>0$, we obtain
\begin{equation} \label{trest2}
\int_{\mathbb{R}^n} \int_{\mathbb{R}^n} \frac{A(x,y)}{|x-y|^{n+2s}} \frac{(u_h(x)-u_h(y))-(u(x)-u(y))}{|h|}(\varphi(x)-\varphi(y))dydx=0,
\end{equation}
which corresponds to formula $(4.3)$ in \cite[Proposition 4.1]{BLS}.
The further proof of Theorem \ref{HiHol} can now be done in almost exactly the same way as in section 4 and 5 of \cite{BLS} by additionally using the bounds (\ref{eq1}) of $A$ when appropriate.
\begin{proof}[Proof of Theorem \ref{modC2sreg}]
	Since $u_0:=u- \overline u_{B_5} \in H^s(B_5 | \mathbb{R}^n)$ also solves the equation 
	$$ L_{A}u_0=0 \text{ weakly in } B_5,$$ we have
	\begin{align*}
	[u]_{C^{\alpha}(B_3)} & = [u_0]_{C^{\alpha}(B_3)} \\
	& \leq C_1 \left (||u_0||_{L^\infty(B_4)} + \int_{B_4} \int_{B_4} \frac{|u_0(x)-u_0(y)|^2}{|x-y|^{n+2s}}dydx  + \int_{\mathbb{R}^n \setminus B_4} \frac{|u_0(y)|}{|y|^{n+2s}}dy \right )\\
	& \leq C_1 \left (C_2(||\nabla^s u_0||_{L^2(B_5)}+||u_0||_{L^2(B_5)}) + ||\nabla^s u_0||_{L^2(B_5)} + C_3 \left (\int_{\mathbb{R}^n \setminus B_4} \frac{|u_0(y)|^2}{|y|^{n+2s}}dy \right )^\frac{1}{2} \right )\\
	& \leq C_1(C_2+1+C_3C_4) \left (||\nabla^s u||_{L^2(B_5)}+||u_0||_{L^2(B_5)} \right )\\
	& \leq C_1(C_2+1+C_3C_4)(1+C_5)||\nabla^s u||_{L^2(B_5)},
	\end{align*}
	where $C_1=C_1(n,s,\lambda,\alpha)$ is given by Theorem \ref{HiHol}, $C_2=C_2(n,s,\lambda)$ is given by Corollary \ref{supest}, $C_3=C_3(n,s)$ is given by (\ref{intpolar1}), $C_4=C_4(n,s)$ is given by Lemma \ref{tailestimate} and $C_5=C_5(n,s)$ is given by the fractional Poincar\'e inequality (Lemma \ref{Poincare}). This proves Theorem \ref{modC2sreg} with $C=C_1(C_2+1+C_3C_4)(1+C_5)$.
\end{proof}
\begin{remark} \normalfont
Theorem \ref{modC2sreg} can also be proved by the following alternative approach.
In the case when $u$ belongs to $L^{\infty}(\mathbb{R}^n)$ and is a weak solution of an inhomogeneous equation of the form $L_A=f$ in $B_4$ with $f \in L^{\infty}(B_4)$, the additional H\"older regularity from Theorem \ref{modC2sreg} can be proved by essentially the same approach used to prove \cite[Theorem 1.1]{NonlocalGeneral}, cf. the lecture notes \cite{NonlocalNotes}. Theorem \ref{modC2sreg} can then be deduced by a cutoff argument similar to the one applied in \cite[Corollary 2.4]{FracLap}.
\end{remark}

\section{The Dirichlet problem}

In what follows, we fix measurable functions $D_i:\mathbb{R}^n \times \mathbb{R}^n \to \mathbb{R}$ $(i=1,...,m)$ that are symmetric and bounded by some $\Lambda>0$.
\begin{prop} \label{Dirichlet}
Consider a kernel coefficient $A \in \mathcal{L}_0(\lambda)$. Let $\Omega \subset \mathbb{R}^n$ be a domain, $g_i,h \in H^s(\Omega | \mathbb{R}^n)$, $f \in L^2(\Omega)$, $b \in L^{\infty}(\Omega)$ and $l:=\essinf_{x \in \Omega} b(x)$. If $\Omega$ is bounded, then we assume that $l \geq 0$, otherwise we assume that $l>0$. Then there exists a unique solution $u \in H^s(\Omega | \mathbb{R}^n)$ of the weak Dirichlet problem
\begin{equation} \label{constcof31}
\begin{cases} \normalfont
L_{A} u +bu = \sum_{i=1}^m L_{D_i} g_i + f & \text{ weakly in } \Omega \\
             u = h & \text{ a.e. in } \mathbb{R}^n \setminus \Omega.
\end{cases}
\end{equation}
Moreover, if $\Omega$ is bounded and $b\equiv0$, then $u$ satisfies the estimate
\begin{equation} \label{gradest8}
||\nabla^s u||_{L^2(\Omega)} \leq C \left (||\nabla^s h||_{L^2(\Omega)} + \sum_{i=1}^m ||\nabla^s g_i||_{L^2(\Omega)} + ||f||_{L^{2}(\Omega)} \right ),
\end{equation}
where $C=C(n,s,\lambda,\Lambda,|\Omega|)$.
\end{prop}

\begin{proof}
Consider the symmetric bilinear form
$$ \mathcal{E}:H_0^s(\Omega | \mathbb{R}^n) \times H_0^s(\Omega | \mathbb{R}^n) \to \mathbb{R}, \quad \mathcal{E}(w,\varphi):=\mathcal{E}_A(w,\varphi) + (bw,\varphi)_{L^2(\Omega)}.$$
First of all, fix some $w \in H_0^s(\Omega | \mathbb{R}^n)$. We have
$$ \mathcal{E}(w,w) \leq \lambda \int_{\mathbb{R}^n} \int_{\mathbb{R}^n} \frac{(w(x)-w(y))^2}{|x-y|^{n+2s}}dydx + ||b||_{L^\infty(\Omega)} ||w||_{L^2(\Omega)}^2 \leq \max\{\lambda,||b||_{L^\infty(\Omega)}\} ||w||_{H^s(\mathbb{R}^n)}^2 .$$
Let us first consider the case when $\Omega$ is unbounded, in this case we have $l>0$ and therefore 
$$
\mathcal{E}(w,w) \geq \lambda^{-1} \int_{\mathbb{R}^n} \int_{\mathbb{R}^n} \frac{(w(x)-w(y))^2}{|x-y|^{n+2s}}dydx + l ||w||_{L^2(\mathbb{R}^n)}^2 \geq C_1 ||w||_{H^s(\mathbb{R}^n)}^2 ,
$$
where $C_1=\min\{\lambda^{-1},l\}>0$.
If $\Omega$ is bounded, then we have $l \geq 0$. Since we have $w=0$ a.e. in $\mathbb{R}^n \setminus \Omega$ and $w \in H^s(\mathbb{R}^n)$, in this case H\"older's inequality and the fractional Sobolev inequality (cf. \cite[Theorem 6.5]{Hitch}) yield
\begin{equation} \label{FPI}
\begin{aligned} 
\int_{\mathbb{R}^n} w^2 dx = \int_{\Omega} w^2 dx & \leq |\Omega|^{\frac{2s}{n}} \left ( \int_{\Omega} |w|^{\frac{2n}{n-2s}} dx \right )^{\frac{n-2s}{n}} \\
& \leq C_2 |\Omega|^{\frac{2s}{n}} \int_{\mathbb{R}^n} \int_{\mathbb{R}^n} \frac{(w(x)-w(y))^2}{|x-y|^{n+2s}}dydx,
\end{aligned}
\end{equation}
where $C_2=C_2(n,s)>0$.
We deduce
\begin{align*}
\mathcal{E}(w,w) & \geq \lambda^{-1} \int_{\mathbb{R}^n} \int_{\mathbb{R}^n} \frac{(w(x)-w(y))^2}{|x-y|^{n+2s}}dydx \\
& \geq \frac{\lambda^{-1}}{2} \left ( \int_{\mathbb{R}^n} \int_{\mathbb{R}^n} \frac{(w(x)-w(y))^2}{|x-y|^{n+2s}}dydx + C_2^{-1}|\Omega|^{-\frac{2s}{n}} \int_{\mathbb{R}^n} w^2 dx \right ) \geq C_3 ||w||_{H^s(\mathbb{R}^n)}^2,
\end{align*}
where $C_3=\frac{\lambda^{-1}}{2} \min \left \{ 1, C_2^{-1}|\Omega|^{-\frac{2s}{n}} \right \}>0$.
We obtain that in both cases $\mathcal{E}(\cdot,\cdot)$ is positive definite and hence an inner product in $H_0^s(\Omega | \mathbb{R}^n)$ that is equivalent to the inner product $(\cdot,\cdot)_{H^s(\mathbb{R}^n)}$ defined in section 3. Therefore $H_0^s(\Omega | \mathbb{R}^n)$ with the inner product $\mathcal{E}(\cdot,\cdot)$ is a Hilbert space. Since moreover by H\"older's inequality the expression 
$$ -\mathcal{E}_{A}(h,\varphi) - (bh,\varphi)_{L^2(\Omega)} + \sum_{i=1}^m \mathcal{E}_{D_i}(g_i,\varphi) + (f,\varphi)_{L^2(\Omega)}$$
is a bounded linear functional of $\varphi \in H_0^s(\Omega | \mathbb{R}^n)$, by the Riesz representation theorem there exists a unique $w \in H_0^s(\Omega | \mathbb{R}^n)$ such that 
\begin{equation} \label{RieszRep}
\begin{aligned}
& \mathcal{E}_{A}(w,\varphi) +(bw,\varphi)_{L^2(\Omega)} \\
= & -\mathcal{E}_{A}(h,\varphi) - (bh,\varphi)_{L^2(\Omega)} + \sum_{i=1}^m \mathcal{E}_{D_i}(g_i,\varphi) + (f,\varphi)_{L^2(\Omega)} \quad \forall \varphi \in H_0^s(\Omega | \mathbb{R}^n).
\end{aligned}
\end{equation}
But then the function $u:=w+h \in H^s(\Omega | \mathbb{R}^n)$ solves the Dirichlet problem (\ref{constcof31}).
Furthermore, if $u$ and $v$ both solve the Dirichlet problem (\ref{constcof31}), then $u-h$ and $v-h$ both satisfy (\ref{RieszRep}), so that by the uniqueness part of the Riesz representation theorem we deduce $u-h=v-h$ a.e. in $\mathbb{R}^n$ and therefore $u=v$ a.e. in $\mathbb{R}^n$, so that the Dirichlet problem (\ref{constcof31}) has a unique solution. \newline
Let us now prove that if $\Omega$ is bounded and $b \equiv 0$, then the unique solution $u \in H_0^s(\Omega | \mathbb{R}^n)$ of (\ref{constcof31}) satisfies the estimate (\ref{gradest8}). In order to accomplish this, note that by (\ref{FPI}) for any $w \in H_0^s(\Omega | \mathbb{R}^n)$ we have
\begin{align*}
\int_{\Omega} |f(x)||w(x)|dx & \leq ||f||_{L^{2}(\Omega)} ||w||_{L^{2}(\Omega)} \\
& \leq C_2^{\frac{1}{2}}|\Omega|^{\frac{s}{n}} ||f||_{L^{2}(\Omega)} \left (\int_{\mathbb{R}^n} \int_{\mathbb{R}^n} \frac{(w(x)-w(y))^2}{|x-y|^{n+2s}}dydx \right )^{\frac{1}{2}} \\
& \leq 2 C_2^{\frac{1}{2}}|\Omega|^{\frac{s}{n}} ||f||_{L^{2}(\Omega)} ||\nabla^s w||_{L^2(\Omega)} .
\end{align*}
Since $w:=u-h \in H_0^s(\Omega | \mathbb{R}^n)$ satisfies (\ref{RieszRep}), using $\varphi=w$ as a test function in (\ref{RieszRep}) along with the Cauchy-Schwarz inequality yields
\begingroup
\allowdisplaybreaks
\begin{align*}
||\nabla^s w||_{L^2(\Omega)}^2 \leq & \int_{\mathbb{R}^n} \int_{\mathbb{R}^n} \frac{(w(x)-w(y))^2}{|x-y|^{n+2s}}dydx \\
\leq & \lambda \int_{\mathbb{R}^n} \int_{\mathbb{R}^n} A(x,y) \frac{(w(x)-w(y))^2}{|x-y|^{n+2s}}dydx \\
= & \lambda \bigg ( - \int_{\mathbb{R}^n} \int_{\mathbb{R}^n} A(x,y) \frac{(h(x)-h(y))(w(x)-w(y))}{|x-y|^{n+2s}}dydx \\
& + \sum_{i=1}^m \int_{\mathbb{R}^n} \int_{\mathbb{R}^n} D_i(x,y) \frac{(g_i(x)-g_i(y))(w(x)-w(y))}{|x-y|^{n+2s}}dydx + \int_{\Omega} f(x)w(x)dx\bigg ) \\
\leq & \lambda \bigg (\lambda \int_{\mathbb{R}^n} \int_{\mathbb{R}^n} \frac{|h(x)-h(y)||w(x)-w(y)|}{|x-y|^{n+2s}}dydx \\
& + \Lambda \sum_{i=1}^m \int_{\mathbb{R}^n} \int_{\mathbb{R}^n} \frac{|g_i(x)-g_i(y)||w(x)-w(y)|}{|x-y|^{n+2s}}dydx + \int_{\Omega} |f(x)||w(x)|dx \bigg ) \\
\leq & 2\lambda \max \{ \lambda, \Lambda, 2C_2^{\frac{1}{2}}|\Omega|^{\frac{s}{n}} \} \bigg (\int_{\Omega} \int_{\mathbb{R}^n} \frac{|h(x)-h(y)||w(x)-w(y)|}{|x-y|^{n+2s}}dydx \\
& + \sum_{i=1}^m \int_{\Omega} \int_{\mathbb{R}^n} \frac{|g_i(x)-g_i(y)||w(x)-w(y)|}{|x-y|^{n+2s}}dydx + ||f||_{L^{2}(\Omega)} ||\nabla^s w||_{L^2(\Omega)} \bigg ) \\
\leq & C_4 ||\nabla^s w||_{L^2(\Omega)} \left ( ||\nabla^s h||_{L^2(\Omega)} + \sum_{i=1}^m ||\nabla^s g_i||_{L^2(\Omega)} + ||f||_{L^{2}(\Omega)}\right ) ,
\end{align*}
where $C_4:=2 \lambda \max \{ \lambda, \Lambda, 2C_2^{\frac{1}{2}}|\Omega|^{\frac{s}{n}} \}$.
\endgroup
We obtain
\begin{align*}
||\nabla^s u||_{L^2(\Omega)} & \leq 2( ||\nabla^s w||_{L^2(\Omega)} + ||\nabla^s h||_{L^2(\Omega)}) \\
& \leq 2 \left ( C_4 \left ( ||\nabla^s h||_{L^2(\Omega)} + \sum_{i=1}^m ||\nabla^s g_i||_{L^2(\Omega)} + ||f||_{L^{2}(\Omega)} \right ) + ||\nabla^s h||_{L^2(\Omega)} \right ) \\
& \leq C \left (||\nabla^s h||_{L^2(\Omega)} + \sum_{i=1}^m ||\nabla^s g_i||_{L^2(\Omega)} + ||f||_{L^{2}(\Omega)} \right ),
\end{align*}
where $C= 2 (C_4+1).$
\end{proof}
For a treatment of the nonlocal Dirichlet problem for a much more general class of kernels, we refer to \cite{Voigt}.

\section{Higher integrabillity of $\nabla^s u$}
For the rest of this paper, we assume that the kernel coefficient $A$ belongs to the class $\mathcal{L}_1(\lambda)$.
\subsection{An approximation argument}
A key step in the proof of the higher integrability of $\nabla^s u$ is given by the following approximation lemma.
\begin{lem} \label{apppl}
Let M be an arbitrary positive real number.
For any $\varepsilon >0$ there exists some $\delta = \delta (\varepsilon,n,s,\lambda,\Lambda,M) >0$, such that for any weak solution $u \in H^s(B_5 | \mathbb{R}^n)$ of the equation 
$$L_A u = \sum_{i=1}^m L_{D_i} g_i + f \text{ in } B_5$$ under the assumptions that
\begin{equation} \label{conddddd}
\dashint_{B_5} |\nabla^s u|^2 dx \leq M
\end{equation}
and that
\begin{equation} \label{condddddd}
\dashint_{B_5} \left ( f^{2}+ \sum_{i=1}^m |\nabla^s g_i|^2 \right ) dx \leq M\delta^2,
\end{equation}
there exists a weak solution $v \in H^s(B_5 | \mathbb{R}^n)$ of the equation
\begin{equation} \label{constcof}
L_{A} v = 0 \text{ in } B_5
\end{equation}
that satisfies
\begin{equation} \label{L2esti}
||\nabla^s(u-v)||_{L^2(B_5)} \leq \varepsilon.
\end{equation}
Moreover, $v$ satisfies the estimate
\begin{equation} \label{localAS}
||\nabla^s v||_{L^\infty(B_2)} \leq N_0 
\end{equation}
for some constant $N_0=N_0(n,s,\lambda,\Lambda,M)$.
\end{lem}

\begin{proof} 
Fix $\varepsilon>0$ and let $\delta>0$ to be chosen.
Let $v \in H^s(B_5 | \mathbb{R}^n)$ be the unique weak solution of the problem
\begin{equation} \label{constcof3}
\begin{cases} \normalfont
L_{A} v = 0 & \text{ weakly in } B_5 \\
             v = u & \text{ a.e. in } \mathbb{R}^n \setminus B_5,
\end{cases}
\end{equation}
note that $v$ exists by Proposition \ref{Dirichlet}.
Observe that we have 
\begin{equation} \label{deffeq1}
\begin{cases} \normalfont
L_{A} (u-v) = \sum_{i=1}^m L_{D_i} g_i + f & \text{ weakly in } B_5 \\
             u-v = 0 & \text{ a.e. in } \mathbb{R}^n \setminus B_5.
\end{cases}
\end{equation}
Thus, by the estimate (\ref{gradest8}) from Proposition \ref{Dirichlet} and (\ref{condddddd}), there exists a constant $C_1=C_1(n,s,\lambda,\Lambda)$ such that
\begin{equation} \label{L2first}
\int_{B_5} |\nabla^s(u - v)|^2 dx \leq C_1 \left ( \sum_{i=1}^m \int_{B_5} |\nabla^s g_i|^2 dx + \int_{B_5} f^{2} dx \right ) \leq C_1 |B_5|M \delta^2 \leq \varepsilon^2 .
\end{equation}
where the last inequality follows by choosing $\delta$ sufficiently small.
This completes the proof of $(\ref{L2esti})$. \newline
Let us now proof the estimate $(\ref{localAS})$. 
For almost every $x \in B_2$, by Corollary \ref{taillinf} we have
$$ \int_{\mathbb{R}^n \setminus B_3} \frac{(v(x)-v(y))^2}{|x-y|^{n+2s}}dy \leq C_2 \int_{B_3} \int_{\mathbb{R}^n} \frac{(v(z)-v(y))^2}{|z-y|^{n+2s}}dydz,$$
where $C_2=C_2(n,s,\lambda)$.
Now choose $\gamma>0$ small enough such that $\gamma < s$ and $s + \gamma <1$.
In view of the assumption that $A \in \mathcal{L}_1(\lambda)$, by Theorem \ref{modC2sreg} we have 
$$[v]_{C^{s+\gamma}(B_3)} \leq C_{3} ||\nabla^s v||_{L^2(B_5)} $$
for some constant $C_{3}=C_{3}(n,s,\lambda,\gamma)$.
Thus, for almost every $x \in B_2$ we have 
\begin{align*}
\int_{B_3} \frac{(v(x)-v(y))^2}{|x-y|^{n+2s}}dy & \leq [v]_{C^{s+\gamma}(B_3)}^2 \int_{B_3} \frac{dy}{|x-y|^{n-2\gamma}} \\
& = C_{4}  [v]_{C^{s+\gamma}(B_3)}^2 \leq C_{4} C_{3}^2 \int_{B_5} \int_{\mathbb{R}^n} \frac{(v(z)-v(y))^2}{|z-y|^{n+2s}}dydz,
\end{align*}
where $C_{4}=C_{4}(n,\gamma)<\infty$.
Applying the estimate (\ref{gradest8}) from Proposition \ref{Dirichlet} to (\ref{constcof3}) yields
\begin{align*}
\int_{B_5} \int_{\mathbb{R}^n} \frac{(v(z)-v(y))^2}{|z-y|^{n+2s}}dydz 
\leq C_{5} \int_{B_5} \int_{\mathbb{R}^n} \frac{(u(z)-u(y))^2}{|z-y|^{n+2s}}dydz,
\end{align*}
where $C_{5}=C_5(n,s,\lambda,\Lambda)$.
By combining the above estimates, along with (\ref{conddddd}) we conclude that
\begin{align*}
(\nabla^s v)^2 (x) & = \int_{\mathbb{R}^n \setminus B_3} \frac{(v(x)-v(y))^2}{|x-y|^{n+2s}}dy + \int_{B_3} \frac{(v(x)-v(y))^2}{|x-y|^{n+2s}}dy\\
& \leq C_2 \int_{B_5} \int_{\mathbb{R}^n} \frac{(v(z)-v(y))^2}{|z-y|^{n+2s}}dydz + C_{4} C_{3}^2 \int_{B_5} \int_{\mathbb{R}^n} \frac{(v(z)-v(y))^2}{|z-y|^{n+2s}}dydz \\
& \leq C_{5} (C_2+C_{4} C_{3}^2) \int_{B_5} \int_{\mathbb{R}^n} \frac{(u(z)-u(y))^2}{|z-y|^{n+2s}}dydz \\
& \leq C_{5} (C_2+C_{4} C_{3}^2) |B_5|M
\end{align*}
for almost every $x \in B_2$, so that (\ref{localAS}) holds with $N_0=(C_{5} (C_2+C_{4} C_{3}^2) |B_5|M)^{\frac{1}{2}}$.
\end{proof}

\subsection{A real variable argument}
We now combine the above approximation lemma with the techniques from section 2.
\begin{lem} \label{mfuse}
There is a constant $N_1=N_1(n,s,\lambda,\Lambda) > 1$, such that the following holds. For any $\varepsilon > 0$ there exists some $\delta = \delta(\varepsilon,n,s,\lambda, \Lambda) > 0$, 
such that for any $z \in \mathbb{R}^n$, any $r \in (0,1]$, any bounded domain $U \subset \mathbb{R}^n$ such that $B_{5r}(z) \subset U$ and any weak solution $ u \in H^s(B_{5r}(z) | \mathbb{R}^n)$ 
of the equation 
$$ L_A u = \sum_{i=1}^m L_{D_i} g_i + f \text { in } B_{5r}(z)$$
with
\begin{align*}
\left \{ x \in B_r(z) \mid \mathcal{M}_{U} (|\nabla^s u|^2)(x) \leq 1 \right \} & \cap \left \{ x \in B_r(z) \mid \mathcal{M}_{U} \left (|f|^{2} + \sum_{i=1}^m |\nabla^s g_i|^2 \right )(x) \leq \delta^2 \right \} \neq \emptyset,
\end{align*}
we have  
\begin{equation} \label{ccll}
\left | \left \{ x \in B_r(z) \mid \mathcal{M}_{U} (|\nabla^s u|^2)(x) > N_1^2 \right \} \right | < \varepsilon |B_r|.
\end{equation}
\end{lem}

\begin{proof}
Let $\theta >0$ and $M>0$ to be chosen and consider the corresponding $\delta = \delta(\theta,n,s,\lambda,\Lambda,M) > 0$ given by Lemma $\ref{apppl}$. 
Fix $r \in (0,1]$ and $z \in \mathbb{R}^n$. For any $x \in U^\prime := \{\frac{x-z}{r} \mid x \in U\}$, define 
\begin{align*}
\widetilde A (x,y):= A (rx+z,ry+z)=A(rx,ry), \quad \widetilde D_i (x,y):= D_i (rx+z,ry+z),\\
\widetilde u (x) := r^{-s} u(rx+z), \quad \widetilde g_i (x) := r^{-s} g_i(rx+z), \quad \widetilde f (x) := r^s f(rx+z)
\end{align*}
and note that under the above assumptions 
$\widetilde A$ belongs to the class $\mathcal{L}_1(\lambda)$ and that $\widetilde u \in H^s(B_5 | \mathbb{R}^n)$ satisfies 
$$ L_{\widetilde A} \widetilde u = \sum_{i=1}^m L_{\widetilde D_i} \widetilde g_i + \widetilde f \text { weakly in } B_5. $$
Hence, by Lemma \ref{apppl} there exists a weak solution $\widetilde v \in H^s(B_5 | \mathbb{R}^n)$ of 
$$ L_{\widetilde A}\widetilde v = 0 \text{ in } B_5 $$
such that 
\begin{equation} \label{apss}
\int_{B_2} |\nabla^s(\widetilde u -\widetilde v)|^2 dx \leq \theta^2,
\end{equation}
provided that the conditions $(\ref{conddddd})$ and $(\ref{condddddd})$ are satisfied.
By assumption, there exists a point $x \in B_r(z)$ such that 
$$\mathcal{M}_{U} (|\nabla^s u|^2)(x) \leq 1, \quad \mathcal{M}_{U} \left (|f|^{2} + \sum_{i=1}^m |\nabla^s g_i|^2 \right )(x) \leq \delta^2.$$
By the scaling and translation invariance of the Hardy-Littlewood maximal function (Lemma $\ref{scale}$), for the point $x_0:=\frac{x-z}{r} \in B_1$ we thus have
$$ \mathcal{M}_{U^\prime} (|\nabla^s \widetilde u|^2)(x_0) = \mathcal{M}_{U} (|\nabla^s u|^2)(x) \leq 1$$
and
$$ \mathcal{M}_{U^\prime} \left (|\widetilde f|^{2} + \sum_{i=1}^m |\nabla^s \widetilde g_i|^2 \right )(x_0) = \mathcal{M}_U \left (r^{2s}|f|^{2} + \sum_{i=1}^m |\nabla^s g_i|^2 \right )(x) \leq \delta^2.$$
Therefore, for any $\rho>0$ we have 
\begin{equation} \label{lum}
\dashint_{B_\rho(x_0)} |\nabla^s \widetilde u|^2 dx \leq 1, \quad \dashint_{B_\rho(x_0)} \left ( |\widetilde f|^{2}+ \sum_{i=1}^m |\nabla^s \widetilde g_i|^2 \right ) dx \leq \delta^2,
\end{equation}
where the values of $\nabla^s \widetilde u$, $\nabla^s \widetilde g_i$ and $\widetilde f$ outside of $U^\prime$ are replaced by $0$, which we also do for the rest of the proof.
Since $B_5 \subset B_6(x_0)$, by $(\ref{lum})$ we have 
$$
\dashint_{B_5} |\nabla^s \widetilde u|^2 dx \leq \frac{|B_6|}{|B_5|} \text{ } \dashint_{B_6(x_0)} |\nabla^s \widetilde u|^2 dx \leq \left (\frac{6}{5} \right )^n
$$
and
$$
\dashint_{B_5} \left ( |\widetilde f|^2 + \sum_{i=1}^m |\nabla^s \widetilde g_i|^2 \right ) dx \leq \frac{|B_6|}{|B_5|} \text{ } \dashint_{B_6(x_0)} \left ( |\widetilde f|^2 + \sum_{i=1}^m |\nabla^s \widetilde g_i|^2 \right ) dx \leq \left (\frac{6}{5} \right )^n \delta^2,
$$
so that we get that $\widetilde u$, $\widetilde g_i$ and $\widetilde f$ satisfy 
the conditions $(\ref{conddddd})$ and $(\ref{condddddd})$ with $M=\left (\frac{6}{5} \right )^n$. Therefore, (\ref{apss}) is satisfied by $\widetilde u$ and the corresponding approximate solution $\widetilde v$. Considering the function $v \in H^s(U|\mathbb{R}^n)$ given by $v(x):=r^{s} \widetilde v \left (\frac{x-z}{r} \right)$ and rescaling back yields 
\begin{equation} \label{apss91}
\int_{B_{2r}(y)} |\nabla^s(u -v)|^2dx = r^n \int_{B_2} |\nabla^s(\widetilde u -\widetilde v)|^2dx \leq \theta^2 r^n.
\end{equation}
By Lemma $\ref{apppl}$, there exists a constant $N_0=N_0(n,s, \lambda,\Lambda) >0$ such that 
\begin{equation} \label{loclinf}
||\nabla^s \widetilde v||_{L^\infty(B_2)}^2 \leq N_0^2 . 
\end{equation}
Next, we define $N_1 := (\max \{ 4 N_0^2, 3^n \})^{1/2} > 1$ and claim that 
\begin{equation} \label{inclusion}
\left \{ x \in B_1 \mid \mathcal{M}_{U^\prime} ( |\nabla^s \widetilde u|^2 )(x) > N_1^2 \right \} \subset  \left \{ x \in B_1 \mid \mathcal{M}_{B_2} ( |\nabla^s(\widetilde u -\widetilde v)|^2 )(x) > N_0^2 \right \}. 
\end{equation}
To see this, assume that 
\begin{equation} \label{menge}
x_1 \in \left \{ x \in B_1 \mid \mathcal{M}_{B_2} ( |\nabla^s(\widetilde u -\widetilde v)|^2 ) (x) \leq N_0^2 \right \}. 
\end{equation}
For $ \rho < 1$, we have $B_\rho (x_1) \subset B_1(x_1) \subset B_2$, so that together with $(\ref{menge})$ and $(\ref{loclinf})$ we deduce 
\begin{align*}
\dashint_{B_\rho (x_1)} |\nabla^s \widetilde u|^2 dx & \leq 2 \text{ } \dashint_{B_\rho(x_1)} \left ( |\nabla^s (\widetilde u -\widetilde v)|^2 + |\nabla^s \widetilde v|^2 \right )dx \\
& \leq 2 \text{ } \dashint_{B_\rho(x_1)} |\nabla^s (\widetilde u -\widetilde v)|^2 dx + 2 \text{ } ||\nabla^s \widetilde v||_{L^\infty(B_\rho(x_1))}^2 \\
& \leq 2 \text{ } \mathcal{M}_{B_2} (|\nabla^s (\widetilde u -\widetilde v)|^2) (x_1) + 2 \text{ } ||\nabla^s \widetilde v||_{L^\infty(B_2)}^2 \leq 4 N_0^2. 
\end{align*}
On the other hand, for $\rho \geq 1$ we have $ B_\rho (x_1) \subset B_{3 \rho}(x_0)$, so that $(\ref{lum})$ implies 
$$
\dashint_{B_\rho(x_1)} |\nabla^s \widetilde u|^2 dx \leq \frac{|B_{3 \rho}|}{|B_\rho|} \dashint_{B_{3 \rho} (x_0)} |\nabla^s \widetilde u|^2 dx \leq 3^n.
$$
Thus, we have $$ x_1 \in \left \{ x \in B_1 \mid \mathcal{M}_{U^\prime}( |\nabla^s \widetilde u|^2) (x) \leq N_1^2 \right \} ,$$ 
which implies $(\ref{inclusion})$. In view of the scaling and translation invariance of the Hardy-Littlewood maximal function (Lemma $\ref{scale}$), (\ref{inclusion}) is equivalent to
\begin{equation} \label{inclusion91}
\left \{ x \in B_r(z) \mid \mathcal{M}_{U} ( |\nabla^s u|^2 )(x) > N_1^2 \right \} \subset  \left \{ x \in B_r(z) \mid \mathcal{M}_{B_{2r}(z)} ( |\nabla^s (u -v)|^2 )(x) > N_0^2 \right \}. 
\end{equation}
For any $\varepsilon > 0$, using $(\ref{inclusion91})$, the weak $1$-$1$ estimate from Theorem \ref{Maxfun} and $(\ref{apss91})$, we conclude that there exists some constant $C=C(n)>0$ such that 
\begin{align*}
\left | \left \{ x \in B_r(z) \mid \mathcal{M}_U ( |\nabla^s u|^2)(x) >N_1^2 \right \} \right | 
& \leq \left |  \left \{ x \in B_r(z) \mid \mathcal{M}_{B_{2r}(z)} ( |\nabla^s (u -v)|^2 )(x) > N_0^2 \right \} \right | \\
& \leq \frac{C}{N_0^2} \int_{B_{2r}(z)} |\nabla^s(u -v)|^2 dx \\
& \leq \frac{C}{N_0^2} \theta^2 r^n < \varepsilon |B_r|,
\end{align*}
where the last inequality is obtained by choosing $\theta$ and thus also $\delta$ sufficiently small. \newline This finishes our proof.
\end{proof}

\begin{remark} \normalfont
Note that in the above proof, the choice of $\theta$ and thus also the choice of a sufficiently small $\delta$ does not depend on the radius $r$, which is due to the fact that $|B_r|=c r^n$ for some constant $c=c(n)>0$. This is vital in our further proof of the $H^{s,p}$ regularity.
\end{remark}
Next, we refine the statement of Lemma $\ref{mfuse}$ in order make it applicable for proving the assumptions of Lemma $\ref{modVitali}$. 
\begin{cor} \label{applic}
There is a constant $N_1=N_1(n,s,\lambda,\Lambda) > 1$, such that the following holds. For any $\varepsilon > 0$ there exists some $\delta = \delta(\varepsilon,n,s,\lambda,\Lambda) > 0$, 
such that for any $z \in B_1$, any $r \in (0,1)$ and any weak solution $ u \in H^s(B_6 | \mathbb{R}^n)$ 
of the equation 
$$ L_A u = \sum_{i=1}^m L_{D_i} g_i + f \text { in } B_6$$
with
\begin{equation} \label{Lvv}
\left | \left \{ x \in B_r(z) \mid \mathcal{M}_{B_6} (|\nabla^s u|^2)(x) > N_1^2 \right \} \cap B_1 \right | \geq \varepsilon |B_r|,
\end{equation}
we have
\begin{equation} \label{Lvv1}
\begin{aligned}
B_r(z) \cap B_1 \subset & \left \{ x \in B_1 \mid \mathcal{M}_{B_6} (|\nabla^s u|^2)(x) > 1 \right \} \\
 & \cup \left \{ x \in B_1 \mid \mathcal{M}_{B_6} \left (|f|^{2} + \sum_{i=1}^m |\nabla^s g_i|^2 \right )(x) > \delta^2 \right \}.
\end{aligned}
\end{equation}
\end{cor}

\begin{proof}
Let $N_1=N_1(n,s,\lambda,\Lambda)>1$ be given by Lemma $\ref{mfuse}$.
Fix $\varepsilon > 0$, $r \in (0,1)$, $z \in \mathbb{R}^n$ and consider the corresponding $\delta= \delta (\varepsilon,n,s,\lambda,\Lambda)>0$ given by Lemma $\ref{mfuse}$.
We argue by contradiction.
Assume that $(\ref{Lvv})$ is satisfied but that $(\ref{Lvv1})$ is false, so that there exists some $x_0 \in B_r(z) \cap B_1$ such that 
\begin{align*} 
x_0 \in B_r(z) & \cap \left \{ x \in B_1 \mid \mathcal{M}_{B_{6}} (|\nabla^s u|^2)(x) \leq 1 \right \} \\
& \cap \left \{ x \in B_1 \mid \mathcal{M}_{B_6} \left (|f|^{2} + \sum_{i=1}^m |\nabla^s g_i|^2 \right )(x) \leq \delta^2 \right \} \\
\subset &  \left \{ x \in B_r(z) \mid \mathcal{M}_{B_6} (|\nabla^s u|^2)(x) \leq 1 \right \} \\
& \cap \left \{ x \in B_r(z) \mid \mathcal{M}_{B_6} \left (|f|^{2} + \sum_{i=1}^m |\nabla^s g_i|^2 \right )(x) \leq \delta^2 \right \} .
\end{align*}
Since moreover we have $B_{5r}(z) \subset B_6$, Lemma $\ref{mfuse}$ with $U=B_6$ yields 
\begin{align*}
& \left | \left \{ x \in B_r(z) \mid \mathcal{M}_{B_6} (|\nabla^s u|^2)(x) > N_1^2 \right \} \cap B_1 \right | \\
\leq & \left | \left \{ x \in B_r(z) \mid \mathcal{M}_{B_6} (|\nabla^s u|^2)(x) > N_1^2 \right \} \right | < \varepsilon |B_r|,
\end{align*}
which contradicts $(\ref{Lvv})$.
\end{proof}

\begin{lem} \label{aplvi}
Let $N_1=N_1(n,s,\lambda,\Lambda) > 1$ be given by Corollary $\ref{applic}$.
Moreover, let $k \in \mathbb{N}$, $\varepsilon \in (0,1)$, set $\varepsilon_1 := 10^n \varepsilon$ 
and consider the corresponding $\delta = \delta(\varepsilon,n,s,\lambda,\Lambda)>0$ given by Corollary $\ref{applic}$.
Then for any weak solution $ u \in H^s(B_6 | \mathbb{R}^n)$ 
of the equation 
$$ L_A u = \sum_{i=1}^m L_{D_i} g_i + f \text { in } B_6$$
with
\begin{equation} \label{air}
\left | \left \{ x \in B_1 \mid \mathcal{M}_{B_6} (|\nabla^s u|^2)(x) > N_1^2 \right \} \right | < \varepsilon |B_1| ,
\end{equation}
we have 
\begingroup
\allowdisplaybreaks
\begin{align*}
& \left | \left \{ x \in B_1 \mid \mathcal{M}_{B_6} (|\nabla^s u|^2)(x) > N_1^{2k} \right \} \right | \\
\leq & \sum_{j=1}^k \varepsilon_1^j \left | \left \{ x \in B_1 \mid \mathcal{M}_{B_6} \left (|f|^{2} + \sum_{i=1}^m |\nabla^s g_i|^2 \right )(x) > \delta^2 N_1^{2(k-j)} \right \} \right | \\ 
& + \varepsilon_1^k \left | \left \{ x \in B_1 \mid \mathcal{M}_{B_6} (|\nabla^s u|^2)(x) > 1 \right \} \right | .
\end{align*}
\endgroup
\end{lem}

\begin{proof}
We proof this Lemma by induction on $k$. 
In view of $(\ref{air})$ and Corollary $\ref{applic}$, the case $k=1$ is a direct consequence of Lemma $\ref{modVitali}$
applied to the sets 
$$ E := \left \{ x \in B_1 \mid \mathcal{M}_{B_6} (|\nabla^s u|^2)(x) > N_1^2 \right \}  $$
and
\begin{align*}
F:= \left \{ x \in B_1 \mid \mathcal{M}_{B_6} (|\nabla^s u|^2)(x) > 1 \right \} & \cup \left \{ x \in B_1 \mid \mathcal{M}_{B_6} \left (|f|^{2} + \sum_{i=1}^m |\nabla^s g_i|^2 \right )(x) > \delta^2 \right \} .
\end{align*}
Next, assume that the conclusion is valid for some $k \in \mathbb{N}$.
Define $\widehat u := u/N_1$, $\widehat g_i := g_i/N_1$ and $\widehat f := f/N_1$. Then $\widehat u$ clearly satisfies  
$$ L_A \widehat u = \sum_{i=1}^m L_{D_i} \widehat g_i + \widehat f \text { weakly in } B_6.$$
Moreover, since $N_1 > 1$ we have 
\begin{align*}
\left | \left \{ x \in B_1 \mid \mathcal{M}_{B_6} (|\nabla^s \widehat u|^2)(x) > N_1^2 \right \} \right | 
& = \left | \left \{ x \in B_1 \mid \mathcal{M}_{B_6} (|\nabla^s u|^2)(x) > N_1^4 \right \} \right | \\
& \leq \left | \left \{ x \in B_1 \mid \mathcal{M}_{B_6} (|\nabla^s u|^2)(x) > N_1^2 \right \} \right |< \varepsilon |B_1|.
\end{align*}
Thus, using the induction assumption yields 
\begin{align*}
& \left | \left \{ x \in B_1 \mid \mathcal{M}_{B_6} (|\nabla^s u|^2)(x) > N_1^{2(k+1)} \right \} \right | \\
= & \left | \left \{ x \in B_1 \mid \mathcal{M}_{B_6} (|\nabla^s \widehat u|^2)(x) > N_1^{2k} \right \} \right | \\
\leq & \sum_{j=1}^k \varepsilon_1^j \left | \left \{ x \in B_1 \mid \mathcal{M}_{B_6} \left (|\widehat f|^{2} + \sum_{i=1}^m |\nabla^s \widehat g_i|^2 \right )(x) > \delta^2 N_1^{2(k-j)} \right \} \right | \\
& + \varepsilon_1^k \left | \left \{ x \in B_1 \mid \mathcal{M}_{B_6} (|\nabla^s \widehat u|^2)(x) > 1 \right \} \right | \\
= & \sum_{j=1}^k \varepsilon_1^j \left | \left \{ x \in B_1 \mid \mathcal{M}_{B_6} \left (|f|^{2} + \sum_{i=1}^m |\nabla^s g_i|^2 \right )(x) > \delta^2 N_1^{2(k+1-j)} \right \} \right | \\
& + \varepsilon_1^k \left | \left \{ x \in B_1 \mid \mathcal{M}_{B_6} (|\nabla^s u|^2)(x) > N_1^2 \right \} \right |.
\end{align*}
Moreover, by using the case $k=1$ we obtain
\begin{align*}
= & \sum_{j=1}^k \varepsilon_1^j \left | \left \{ x \in B_1 \mid \mathcal{M}_{B_6} \left (|f|^{2} + \sum_{i=1}^m |\nabla^s g_i|^2 \right )(x) > \delta^2 N_1^{2(k+1-j)} \right \} \right | \\
& + \varepsilon_1^k \left | \left \{ x \in B_1 \mid \mathcal{M}_{B_6} (|\nabla^s u|^2)(x) > N_1^2 \right \} \right | \\
\leq & \sum_{j=1}^k \varepsilon_1^j \left | \left \{ x \in B_1 \mid \mathcal{M}_{B_6} \left (|f|^{2} + \sum_{i=1}^m |\nabla^s g_i|^2 \right )(x) > \delta^2 N_1^{2(k+1-j)} \right \} \right | \\
& + \varepsilon_1^k \Bigg ( \varepsilon_1 \left | \left \{ x \in B_1 \mid \mathcal{M}_{B_6} \left (|f|^{2} + \sum_{i=1}^m |\nabla^s g_i|^2 \right )(x) > \delta^2 \right \} \right | \\
& + \varepsilon_1 \left | \left \{ x \in B_1 \mid \mathcal{M}_{B_6} (|\nabla^s u|^2)(x) > 1 \right \} \right | \Big ) \\
= & \sum_{j=1}^{k+1} \varepsilon_1^j \Bigg ( \left | \left \{ x \in B_1 \mid \mathcal{M}_{B_6} \left (|f|^{2} + \sum_{i=1}^m |\nabla^s g_i|^2 \right )(x) > \delta^2 N_1^{2(k+1-j)} \right \} \right | \\
& + \varepsilon_1^{k+1} \left | \left \{ x \in B_1 \mid \mathcal{M}_{B_6} (|\nabla^s u|^2)(x) > 1 \right \} \right | ,
\end{align*}
so that by combining the last two displays we see that the conclusion is valid for $k+1$, which completes the proof.
\end{proof}

We are now set to prove the higher integrability of $\nabla^s u$ in the case of balls. The approach to the proof can be summarized as follows. First of all, we consider an appropriately scaled version of $u$ that satisfies the condition (\ref{air}) from Lemma \ref{aplvi} and also corresponding scaled versions of $g_i$ and $f$. Then we use Lemma \ref{Caff} in order to derive from Lemma \ref{aplvi} the desired $L^p$ estimate in terms of the Hardy-Littlewood maximal functions of the scaled versions of $u$, $g_i$ and $f$, which in view of the strong p-p estimates from Theorem \ref{Maxfun} and rescaling then yields the desired $L^p$ estimate for $\nabla^s u$.
\begin{thm} \label{mainint1}
Let $2<p< \infty$, $g_i \in H^{s,p}(B_6 | \mathbb{R}^n)$ and $f \in L^p(B_6)$. 
If $A$ belongs to $\mathcal{L}_1(\lambda)$ and if all $D_i$ are symmetric and bounded by $\Lambda>0$, then for any weak solution $u \in H^s(B_6 | \mathbb{R}^n)$ 
of the equation 
$$ L_A u = \sum_{i=1}^m L_{D_i} g_i  + f \text{ in } B_6$$ 
we have $\nabla^s u \in L^p(B_1)$. 
Moreover, there exists a constant $C= C(p,n,s,\lambda,\Lambda) >0$ such that
\begin{equation} \label{Lpest98}
||\nabla^s u||_{L^p(B_1)} \leq C \left (||f + \sum_{i=1}^m \nabla^s {g_i}||_{L^p(B_6)} + ||\nabla^s u||_{L^2(B_6)} \right ). 
\end{equation}
\end{thm}

\begin{proof}
Fix $p>2$ and let $N_1=N_1(n,s,\lambda,\Lambda) > 1$ be given by Lemma $\ref{aplvi}$. Moreover, select $\varepsilon \in (0,1)$ such that 
\begin{equation} \label{dwn}
N_1^p 10^n \varepsilon \leq \frac{1}{2}.
\end{equation}
Consider also the corresponding $\delta = \delta(\varepsilon,n,s,\lambda,\Lambda)>0$ given by Corollary $\ref{applic}$.
If $\nabla^s u=0$ a.e$.$ in $B_6$, then the assertion is trivially satisfied, so that we can assume $||\nabla^s u||_{L^2(B_6)} > 0$.
Next, we define 
$$\widehat{u} := \frac{\gamma u}{||\nabla^s u||_{L^2(B_6)}}, \text{ } \widehat{g_i} := \frac{\gamma g_i}{||\nabla^s u||_{L^2(B_6)}} \text{ and } \widehat{f} := \frac{\gamma f}{||\nabla^s u||_{L^2(B_6)}},$$ 
where $\gamma >0$ remains to be chosen independently of $u$, $g_i$ and $f$,
note that we have 
$$ L_A \widehat{u} = \sum_{i=1}^m L_{D_i} \widehat{g_i} + \widehat{f} \text{ weakly in } B_6.$$
Moreover, we have
$$ \int_{B_6} |\nabla^s \widehat{u}|^2 dx = \gamma^2.$$
Combining this observation with the weak $1$-$1$ estimate from Theorem \ref{Maxfun}, it follows that there is a constant $C_1=C_1(n)>0$ such that 
$$\left | \left \{ x \in B_1 \mid \mathcal{M}_{B_6} (|\nabla^s \widehat{u}|^2)(x) > N_1^2 \right \} \right | \leq \frac{C_1}{N_1^2} \int_{B_6} |\nabla^s \widehat{u}|^2 dx 
= \frac{C_1 \gamma^2}{N_1^2} < \varepsilon |B_1|, $$
where the last inequality is obtained by choosing $\gamma$ small enough.
Therefore, all assumptions made in Lemma $\ref{aplvi}$ are satisfied by $\widehat u$.
Furthermore, in view of Lemma $\ref{Caff}$ with $\tau=\delta^2$, $\beta=N_1^2$ and with $p$ replaced by $p/2$, and also taking into account the strong $p$-$p$ estimates for the Hardy-Littlewood maximal function (cf. Theorem \ref{Maxfun}), we deduce that there exist constants 
$C_2=C_2(n,s, \lambda, \Lambda, p)>0$ and $C_3=C_3(n,p)>0$ such that 
\begin{equation} \label{scn}
\begin{aligned}
& \sum_{k=1}^\infty N_1^{pk} \left | \left \{ x \in B_1 \mid \mathcal{M}_{B_6} \left (|\widehat f|^{2} + \sum_{i=1}^m |\nabla^s \widehat g_i|^2 \right )(x) > \delta^2 N_1^{2k} \right \} \right | \\
\leq & C_2 ||\mathcal{M}_{B_6} \left (|\widehat f|^{2} + \sum_{i=1}^m |\nabla^s \widehat g_i|^2 \right )||_{L^{p/2}(B_6)}^{p/2} \\
\leq & C_2 C_3^{p} ||\widehat f + \sum_{i=1}^m \nabla^s \widehat{g_i}||_{L^p(B_6)}^{p}.
\end{aligned}
\end{equation}
Setting $\varepsilon_1 := 10^n \varepsilon$, by $(\ref{dwn})$ we see that 
\begin{equation} \label{Slh}
\sum_{j=1}^\infty (N_1^p \varepsilon_1)^{j} \leq \sum_{j=1}^\infty \left ( \frac{1}{2} \right )^{j} = 1.
\end{equation}
Using Lemma $\ref{aplvi}$, the Cauchy product, $(\ref{Slh})$, $(\ref{scn})$, and setting 
$C_4 := C_2 C_3^{p}$, we compute 
\begingroup
\allowdisplaybreaks
\begin{align*} 
& \sum_{k=1}^\infty N_1^{pk} \left | \left \{ x \in B_1 \mid \mathcal{M}_{B_6} (|\nabla^s \widehat{u}|^2)(x) > N_1^{2k} \right \} \right | \\
\leq & \sum_{k=1}^\infty N_1^{pk} \Bigg ( \sum_{j=1}^k \varepsilon_1^j \left 
| \left \{ x \in B_1 \mid \mathcal{M}_{B_6} \left (|\widehat f|^{2}+\sum_{i=1}^m |\nabla^s \widehat{g_i}|^2 \right )(x)> \delta^2 N_1^{2(k-j)} \right \} \right | \\
& + \varepsilon_1^k \left | \left \{ x \in B_1 \mid \mathcal{M}_{B_6} (|\nabla^s \widehat{u}|^2)(x) > 1 \right \} \right | \Bigg ) \\
= & \left ( \sum_{k=0}^\infty N_1^{pk} \left | \left \{ x \in B_1 \mid \mathcal{M}_{B_6} \left (|\widehat f|^{2}+\sum_{i=1}^m |\nabla^s \widehat{g_i}|^2 \right )(x) > \delta^2 N_1^{2k} \right \} \right | \right) 
\left ( \sum_{j=1}^\infty (N_1^p \varepsilon_1)^{j} \right ) \\
& + \left ( \sum_{k=1}^\infty (N_1^p \varepsilon_1)^{k} \right ) \left | \left \{ x \in B_1 \mid \mathcal{M}_{B_6} (|\nabla^s \widehat{u}|^2)(x) > 1 \right \} \right | \\
\leq & \left ( \sum_{k=1}^\infty N_1^{pk} \left | \left \{ x \in B_1 \mid \mathcal{M}_{B_6} \left (|\widehat f|^{2}+\sum_{i=1}^m |\nabla^s \widehat{g_i}|^2 \right )(x) > \delta^2 N_1^{2k} \right \} \right | + 2 |B_1| \right ) 
\left ( \sum_{j=1}^\infty (N_1^p \varepsilon_1)^{j} \right ) \\
\leq & \sum_{k=1}^\infty N_1^{pk} \left | \left \{ x \in B_1 \mid \mathcal{M}_{B_6} \left (|\widehat f|^{2}+\sum_{i=1}^m |\nabla^s \widehat{g_i}|^2 \right )(x) > \delta^2 N_1^{2k} \right \} \right | + 2 |B_1| \\
\leq & C_4  ||\widehat f + \sum_{i=1}^m \nabla^s \widehat{g_i}||_{L^p(B_6)}^{p} + 2 |B_1| .
\end{align*}
\endgroup
Therefore, by Theorem \ref{Maxfun} and Theorem $\ref{Caff}$ we find that there exists another constant $C_5=C_5(n,s, \lambda, \Lambda, p)>0$ such that 
\begingroup
\allowdisplaybreaks
\begin{align*}
||\nabla^s \widehat{u}||^p_{L^p(B_1)} & \leq ||\mathcal{M}_{B_6}(|\nabla^s \widehat{u}|^2)||_{L^{p/2}(B_1)}^{p/2} \\
& \leq C_5 \left ( \sum_{k=1}^\infty N_1^{pk} \left | \left \{ x \in B_1 \mid \mathcal{M}_{B_6} (|\nabla^s \widehat{u}|^2)(x) > N_1^{2k} \right \} \right | 
+ |B_1| \right ) \\ 
& \leq C_5 \left ( C_4\left ( ||\widehat f||_{L^p(B_6)}^{p} + \sum_{i=1}^m ||\nabla^s \widehat{g_i}||_{L^p(B_6)}^{p} \right ) + 3 |B_1| \right ) \\ 
& \leq C_6^p \left ( ||\widehat f + \sum_{i=1}^m \nabla^s \widehat{g_i}||_{L^p(B_6)}^{p} + 1 \right ), 
\end{align*} 
\endgroup
where $C_6 := \left ( C_5 \max \left \{ C_4, 3 |B_1| \right \} \right )^{1/p} $.
It follows that 
$$ ||\nabla^s \widehat{u}||_{L^p(B_1)} \leq C_6 \left ( ||\widehat f + \sum_{i=1}^m \nabla^s \widehat{g_i}||_{L^p(B_6)}^{p} +1 \right )^{1/p} 
\leq C_6 \left ( ||\widehat f + \sum_{i=1}^m \nabla^s \widehat{g_i}||_{L^p(B_6)} +1 \right ), $$
so that
\begin{align*}
||\nabla^s u||_{L^p(B_1)} & \leq C_6 \left ( ||f + \sum_{i=1}^m \nabla^s {g_i}||_{L^p(B_6)} + \frac{||\nabla^s u||_{L^2(B_6)}}{\gamma} \right ) \\
& \leq C_6 \gamma^{-1} \left ( || f + \sum_{i=1}^m \nabla^s {g_i}||_{L^p(B_6)}+ ||\nabla^s u||_{L^2(B_6)} \right ),
\end{align*}
which proves (\ref{Lpest98}) with $C:= C_6 \gamma^{-1}$.
\end{proof}

\section{Proofs of the main results}
In order to state our main result on local regularity in an optimal way, we define the following notion of local weak solutions. 
\begin{defin}
Let $\Omega \subset \mathbb{R}^n$ be a domain. Given $b \in L^\infty_{loc}(\Omega)$, $f \in L^2_{loc}(\Omega)$ and $g_i \in H^s_{loc}(\Omega | \mathbb{R}^n)$, we say that $u \in H^s_{loc}(\Omega | \mathbb{R}^n)$ is a local weak solution to the equation $L_A u + bu = \sum_{i=1}^m L_{D_i} g_i+ f$ in $\Omega$, if 
$$ \mathcal{E}_A(u,\varphi) + (bu,\varphi)_{L^2(\Omega)}= \sum_{i=1}^m \mathcal{E}_{D_i}(g_i,\varphi) + (f,\varphi)_{L^2(\Omega)} \quad \forall \varphi \in H^s_c(\Omega | \mathbb{R}^n),$$
where by $H^s_c(\Omega | \mathbb{R}^n)$ we denote the set of all functions that belong to $H^s(\Omega | \mathbb{R}^n)$ and are compactly supported in $\Omega$.
\end{defin}
In view of the inclusions $$ H^{s,p}(\mathbb{R}^n) \subset H^{s,p}_{loc}(\Omega | \mathbb{R}^n) \subset H^{s,p}_{loc}(\Omega) \subset W^{s,p}_{loc}(\Omega)$$ for $p \in [2,\infty)$ which we discussed in section 3, Theorem \ref{mainint5} follows directly from the following slightly stronger result in terms of the spaces $H^{s,p}_{loc}(\Omega | \mathbb{R}^n)$ defined in section 3.
\begin{thm} \label{mainint3}
Let $\Omega \subset \mathbb{R}^n$ be a domain, $p \in (2,\infty)$, $s \in (0,1)$, $b \in L^\infty_{loc}(\Omega)$, $g_i \in H^{s,p}_{loc}(\Omega | \mathbb{R}^n)$ and $f \in L^{p_\star}_{loc}(\Omega)$, where $p_\star=\max \left \{\frac{pn}{n+ps},2 \right \}$. 
If $A$ belongs to $\mathcal{L}_1(\lambda)$ and if all $D_i$ are symmetric and bounded by $\Lambda>0$, then for any local weak solution $u \in H^s_{loc}(\Omega | \mathbb{R}^n)$ 
of the equation
\begin{equation} \label{nonloceq1}
L_A u + bu= \sum_{i=1}^m L_{D_i} g_i + f \text{ in } \Omega
\end{equation}
we have $u \in H^{s,p}_{loc}(\Omega | \mathbb{R}^n)$. 
\end{thm}

\begin{proof}
Fix $p \in (2,\infty)$. We first prove the result under the stronger assumption that $f \in L^p_{loc}(\Omega)$.
Fix relatively compact bounded open sets $U \subset \subset V \subset \subset \Omega$. Moreover, fix a smooth domain $U_\star$ such that $U \subset \subset U_\star \subset \subset V$. 
Let $\widetilde f:=f-bu$, so that $u$ is a local weak solution of
\begin{equation} \label{locweak}
L_A u = \sum_{i=1}^m L_{D_i} g_i + \widetilde f \text{ in } \Omega .
\end{equation}
In particular, $u$ is a weak solution of (\ref{locweak}) in $V$.
For any $z \in V$, fix some small enough $r_z \in (0,1)$ such that $B_{6r_z}(z) \subset \subset V$.
Define 
\begin{align*}
A_z (x,y):= A \left (r_z x+z, r_z y+z \right )=A \left (r_z x, r_z y \right ), \quad {D_i}_z (x,y):= D_i \left (r_z x+z, r_z y+z \right ), \\
u_z (x) := r^{-s}_z u \left (r_z x +z \right ), \quad {g_i}_z (x) := r^{-s}_z g \left (r_z x +z \right ), \quad \widetilde f_z (x) := r^s_z \widetilde f \left (r_z x +z \right )
\end{align*} 
and note that for any $z \in V$, 
\begingroup
\allowdisplaybreaks
$A_z$ belongs to the class $\mathcal{L}_1(\lambda)$ and that $u_z$ satisfies 
$$ L_{A_z} u_z =\sum_{i=1}^m L_{{D_i}_z} {g_i}_z + \widetilde f_z \text { weakly in } B_6.$$
Using Theorem $\ref{mainint1}$, for any $q \in (2,\infty)$ we obtain the estimate 
\begin{align*}
& ||\nabla^s u||_{L^q \left (B_{r_z}(z) \right)} = r_z^{\frac{n}{q}} ||\nabla^s u_z||_{L^q(B_1)} \\
  \leq & r_z^{\frac{n}{q}} C_1 
\left (||\widetilde f_z + \sum_{i=1}^m \nabla^s {g_i}_z||_{L^q(B_6)} + ||\nabla^s u_z||_{L^2(B_6)} \right ) \\
= & C_1 \left (||r^s_z \widetilde f + \sum_{i=1}^m \nabla^s g_i||_{L^q(B_{6r_z}(z))} + r_z^{\frac{n}{q}-\frac{n}{2}} ||\nabla^s u||_{L^2(B_{6r_z}(z))} \right ) \\
\leq & C_1 \max \{1, r_z^{\frac{n}{q}-\frac{n}{2}} \} \left (||\widetilde f||_{L^q(B_{6r_z}(z))} + ||\sum_{i=1}^m \nabla^s g_i||_{L^q(B_{6r_z}(z))} + ||\nabla^s u||_{L^2(B_{6r_z}(z))} \right ),
\end{align*}
\endgroup
where $C_1=C_1(q,n,s,\lambda,\Lambda) >0$.
Since $\left \{ B_{r_z}(z) \right \}_{z \in \overline U_\star} $ is an open covering of $\overline U_\star$ and $\overline U_\star$ is compact, there is a finite subcover 
$ \left \{ B_{r_{z_i}}(z_i) \right \}_{i =1}^k$ of $\overline U_\star$ and hence of $U_\star$. 
Let $\{\phi_i \}_{i=1}^k$ be a partition of unity subordinate to the covering $\left \{ B_{r_{z_i}}(z_i) \right \}_{i =1}^k$ of $\overline U_\star$, that is, the $\phi_i$ are non-negative functions on $\mathbb{R}^n$, we have 
$ \phi_i \in C_0^\infty(B_{r_{x_i}}(x_i))$ for all $i=1,...,k$, $\sum_{i=1}^k \phi_j \equiv 1$ in an open neighbourhood of $\overline U_\star$ and $\sum_{i=1}^k \phi_j \leq 1$ in $\mathbb{R}^n$.
Setting $C_2:=C_1 \max \{1, \max_{i=1,...,k} r_{z_i}^{\frac{n}{q}-\frac{n}{2}} \}$ and summing the above estimates over $i=1,...,k$, we conclude 
\begin{align*} 
||\nabla^s u||_{L^q(U_\star)} & =||\sum_{i=1}^k |\nabla^s u| \phi_i ||_{L^q(U_\star)} \\
& \leq \sum_{i=1}^k |||\nabla^s u| \phi_i||_{L^q(B_{r_{z_i}}(z_i))} \\
& \leq \sum_{i=1}^k ||\nabla^s u||_{L^q(B_{r_{z_i}(z_i) })} \\
& \leq \sum_{i=1}^k C_2 \left (||\widetilde f||_{L^q(B_{6r_z}(z))} + ||\sum_{i=1}^m \nabla^s g_i||_{L^q(B_{6r_z}(z))} + ||\nabla^s u||_{L^2(B_{6r_z}(z))} \right ) \\
& \leq \sum_{i=1}^k C_2 \left (||\widetilde f||_{L^q(V)} + ||\sum_{i=1}^m \nabla^s g_i||_{L^q(V)} + ||\nabla^s u||_{L^2(V)} \right ) \\
& = C_2 k \left (||\widetilde f||_{L^q(V)} + ||\sum_{i=1}^m \nabla^s g_i||_{L^q(V)} + ||\nabla^s u||_{L^2(V)} \right ),
\end{align*} 
which implies that for any $q \in (2,\infty)$ we have
\begin{equation} \label{loworderterm1}
||\nabla^s u||_{L^q(U_\star)} \leq C_3 \left (||f ||_{L^q(V)} + || u ||_{L^q(V)} + \sum_{i=1}^m ||\nabla^s {g_i}||_{L^q(V)} + ||\nabla^s u||_{L^2(V)} \right ), 
\end{equation}
where $C_3=C_2 k \max \{1,||b||_{L^\infty(V)} \}$. In particular, since by assumption and Theorem \ref{altcharBessel} we have $f,\nabla^s g_i \in L^p(V)$, for any $q \in [2,p]$ we have $\nabla^s u \in L^q(U_\star)$ whenever $u \in L^q(V)$.
For any $r \in [1,p]$, define
$$ r^\star := \begin{cases} 
\min \{\frac{rn}{n-rs},p\}, & \text{if } rs < n \\
p, & \text{if } rs \geq n ,
\end{cases}$$ 
note that $r^\star \in [1,p]$.
By the embedding theorem of Bessel potential spaces (Theorem \ref{BesselEmbedding}), for any $r \geq 1$ we have
$$ H^{s,r}(U_\star) \hookrightarrow L^{r^\star}(U_\star).$$
Since $u \in H^s(V)$, we have $u \in L^{2^\star}(V)$ and therefore $\nabla^s u \in L^{2^\star}(U_\star)$.
If $p = 2^\star$, we have $u \in L^{p}(U_\star)$, $\nabla^s u \in L^{p}(U_\star)$ and therefore $u \in H^{s,p}(U_\star | \mathbb{R}^n).$ If $p> 2^\star$, then we have $u,\nabla^s_{U_\star} u \in L^{2^\star}(U_\star)$, so that Theorem \ref{altcharBessel} yields $u \in H^{s,2^\star}(U_\star).$ We therefore arrive at $u \in L^{{2^\star}^\star}(U_\star)$. By replacing $U_\star$ with an arbitrary relatively compact smooth open subset of $U_\star$ which contains $U$ if necessary, we therefore obtain $\nabla^s u \in L^{{2^\star}^\star}(U_\star)$. If ${2^\star}^\star = p$, then we have 
$u,\nabla^s_{U_\star} u \in L^{p}(U_\star)$ and therefore $u \in H^{s,p}(U_\star | \mathbb{R}^n).$ If ${2^\star}^\star > p$, then iterating the above procedure also yields $u \in H^{s,p}(U_\star | \mathbb{R}^n)$ and therefore $u \in H^{s,p}(U | \mathbb{R}^n)$ at some point. Since $U$ is an arbitrary relatively compact open subset of $\Omega$, we conclude that $u \in H^{s,p}_{loc}(\Omega | \mathbb{R}^n)$. This finishes the proof when $f \in L^p_{loc}(\Omega)$. 

Next, consider the general case when $f \in L^{p_\star}_{loc}(\Omega)$, where $p_\star = \max \left \{\frac{pn}{n+ps},2 \right \}$. Define the function $f_{\Omega}:\mathbb{R}^n \to \mathbb{R}$ by
$$ f_{\Omega}(x) := \begin{cases} 
f(x), & \text{if } x \in \Omega \\
0, & \text{if } x \in \mathbb{R}^n \setminus \Omega 
\end{cases}$$ 
and note that $f_{\Omega} \in L^{p_\star}(\mathbb{R}^n) \cap L^2(\mathbb{R}^n)$. 
By Proposition \ref{Dirichlet}, there exists a unique weak solution $g \in H^s(\mathbb{R}^n) \subset H^s_{loc}(\Omega | \mathbb{R}^n)$ of the equation
\begin{equation} \label{helpeq}
(-\Delta)^s g + g = f_\Omega \quad \text{in } \mathbb{R}^n,
\end{equation}
where
$$ (-\Delta)^s g(x) = C_{n,s} \int_{\mathbb{R}^n} \frac{g(x)-g(y)}{|x-y|^{n+2s}}dy$$
is the fractional Laplacian.
In view of the classical $H^{2s,p}$ regularity for the fractional Laplacian on the whole space $\mathbb{R}^n$ (cf. for example \cite[Lemma 3.5]{KassMengScott}), we have $g \in H^{2s,p_\star}(\mathbb{R}^n) \hookrightarrow H^{s,p}(\mathbb{R}^n)$ and therefore in particular $g \in H^{s,p}_{loc}(\Omega | \mathbb{R}^n)$. Since furthermore $u$ is a local weak solution of  
$$ L_A u +bu = \left (\sum_{i=1}^m L_{D_i} g_i + (-\Delta)^s g \right ) + g \text{ in } \Omega,$$
by the first part of the proof we obtain that $u \in H^{s,p}_{loc}(\Omega | \mathbb{R}^n)$. This finishes the proof.
\end{proof}

\begin{proof}[Proof of Theorem \ref{mainint0}]
Fix $p \in (2,\infty)$ and let $\delta=\delta(p,n,s,\lambda,\Lambda) >0$ be given by Theorem \ref{mainint1}. We first prove the result under the stronger assumption that $f \in L^2(\mathbb{R}^n) \cap L^p(\mathbb{R}^n)$.
For any $k \in \mathbb{Z}^n$, let $E_k:=B_{\sqrt{n}}(k)$ and $F_k:=B_{2\sqrt{n}}(k)$. We then have $\mathbb{R}^n = \bigcup_{k \in \mathbb{Z}^n} E_k$, moreover, there exists some $N \in \mathbb{N}$ depending only on $n$ such that no point in $\mathbb{R}^n$ is contained in more than $N$ of the balls $F_k$. In other words, we have $\sum_{k \in \mathbb{Z}^d} \chi_{F_k} \leq N$, where $\chi_{F_k}$ is the characteristic function of $F_k$.
\begingroup
\allowdisplaybreaks
Since for $\widetilde f:=f-bu$ we have 
$$ L_A u = \sum_{i=1}^m L_{D_i} g_i + \widetilde f \text{ weakly in } \mathbb{R}^n,$$ 
by the same argument as in the proof of Theorem \ref{mainint3}
for any $k \in \mathbb{Z}^d$ and any $q \in (2,\infty)$ we have 
$$ ||\nabla^s u||_{L^q(E_k)} \leq C \left (||\widetilde f||_{L^q(F_k)} + ||\sum_{i=1}^m \nabla^s g_i||_{L^q(F_k)} + ||\nabla^s u||_{L^2(F_k)} \right ) $$
for some constant $C= C(n,s,q,\lambda,\Lambda) >0$.
It follows that
\begin{align*}
& \int_{\mathbb{R}^n} |\nabla^s u (x)|^q dx
\leq \sum_{k \in \mathbb{Z}^n} \int_{E_k} |\nabla^s u(x)|^q dx \\
\leq & C^q \sum_{k \in \mathbb{Z}^n} \left ( \left ( \int_{F_k} |\widetilde f(x)|^q dx \right )^{\frac{1}{q}} + \left (\int_{F_k} \sum_{i=1}^m |\nabla^s g_i(x)|^q dx \right )^{\frac{1}{q}} + \left ( \int_{F_k} |\nabla^s u(x)|^2 dx \right )^{\frac{1}{2}} \right )^q \\
\leq & C_1 C^q \left ( \left ( \sum_{k \in \mathbb{Z}^n} \int_{F_k} |\widetilde f(x)|^q + \sum_{i=1}^m |\nabla^s g_i(x)|^q dx \right ) + \sum_{k \in \mathbb{Z}^n} \left ( \int_{F_k} |\nabla^s u(x)|^2 dx \right )^{\frac{q}{2}} \right ) \\
\leq & C_1 C^q \left ( \left ( \sum_{k \in \mathbb{Z}^n} \int_{F_k} |\widetilde f(x)|^q + \sum_{i=1}^m |\nabla^s g_i(x)|^q dx \right ) + \left ( \sum_{k \in \mathbb{Z}^n} \int_{F_k} |\nabla^s u(x)|^2 dx \right )^{\frac{q}{2}} \right ) \\
= & C_1 C^q \left ( \left ( \int_{\mathbb{R}^n} \left ( |\widetilde f(x)|^q + \sum_{i=1}^m |\nabla^s g_i(x)|^q \right ) \sum_{k \in \mathbb{Z}^d} \chi_{F_k}(x) dx \right ) + \left ( \int_{\mathbb{R}^n} |\nabla^s u(x)|^2 \sum_{k \in \mathbb{Z}^d} \chi_{F_k}(x) dx \right )^{\frac{q}{2}} \right ) \\
\leq & N^{\frac{q}{2}} C_1 C^q \left ( \left ( \int_{\mathbb{R}^n} |\widetilde f(x)|^q + \sum_{i=1}^m |\nabla^s g_i(x)|^q dx \right ) + \left ( \int_{\mathbb{R}^n} |\nabla^s u(x)|^2 dx \right )^{\frac{q}{2}} \right ) ,
\end{align*}
\endgroup
where $C_1=C_1(q)>0$.
This implies that for any $q \in (2,\infty)$ we have
\begin{equation} \label{loworderterm}
||\nabla^s u||_{L^q(\mathbb{R}^n)} \leq C_2 \left (|| f ||_{L^q(\mathbb{R}^n)} + || u ||_{L^q(\mathbb{R}^n)} + \sum_{i=1}^m ||\nabla^s {g_i}||_{L^q(\mathbb{R}^n)} + ||\nabla^s u||_{L^2(\mathbb{R}^n)} \right ), 
\end{equation}
where $C_2:=N^{\frac{1}{2}} C_1^{\frac{1}{q}} C \max\{1,||b||_{L^\infty(\mathbb{R}^n)}\}$. In particular, since for any $q \in [2,p]$ we have $L^2(\mathbb{R}^n) \cap L^p(\mathbb{R}^n) \hookrightarrow L^q(\mathbb{R}^n)$ and in view of the assumptions and Theorem \ref{altcharBessel} we have $f,\nabla^s g_i \in L^2(\mathbb{R}^n) \cap L^p(\mathbb{R}^n)$, for any $q \in [2,p]$ it follows that $\nabla^s u \in L^q(\mathbb{R}^n)$ whenever $u \in L^q(\mathbb{R}^n)$.
The proof in the case when  $f \in L^2(\mathbb{R}^n) \cap L^p(\mathbb{R}^n)$ can now be concluded by using essentially the same iteration argument as the one in the proof of Theorem \ref{mainint3}. The general case when $f \in L^2(\mathbb{R}^n) \cap L^{p_\star}(\mathbb{R}^n)$ then can once again be treated by solving the equation (\ref{helpeq}) under optimal regularity, as we did in the proof of Theorem \ref{mainint3}.
\end{proof}

\begin{remarkb} \normalfont
An interesting question is if it is possible to prove a global $H^{s,p}$ regularity result in smooth enough bounded domains $\Omega$ corresponding to our local regularity result Theorem \ref{mainint3}. Our approach is based on a $C^{s+\gamma}$ estimate $(\gamma>0)$ for nonlocal equations with translation invariant kernels, however it is known that already in the case of the fractional Laplacian in a unit ball the optimal regularity up to the boundary is $C^s(\overline {B_1})$, cf. \cite[section 7.1]{NonlocalSurvey}. Therefore, at least with our methods proving such a global $H^{s,p}$ regularity result for the equations we consider in this work seems to be unattainable even in the case when $\Omega$ is very regular.
\end{remarkb}

\bibliographystyle{amsplain}

\end{document}